\documentclass[12pt,a4paper]{article}
\usepackage[utf8]{inputenc}
\usepackage[english]{babel}
\usepackage[ruled,section]{algorithm}
\usepackage{makeidx}
\usepackage[font=small]{caption}
\usepackage{amsmath}
\usepackage{amsthm}
\usepackage{amsfonts}
\usepackage{amssymb}
\usepackage{mathrsfs}
\usepackage{graphicx}
\usepackage[margin=25mm]{geometry}
\usepackage{enumerate}
\usepackage[T1]{fontenc}
\usepackage{mathtools}
\usepackage{nicefrac}
\usepackage{braket}
\usepackage{subfig}
\usepackage{csquotes}
\usepackage[plainpages=false,pagebackref=false]{hyperref}   
\usepackage{color}
\usepackage{authblk}
\usepackage{hyperref}
\makeindex
\hypersetup{pdfpagelabels,hyperindex,colorlinks=true,breaklinks=true,bookmarks=true,linkcolor=black,citecolor=black,
urlcolor=black}
\providecommand{\abs}[1]{\left|#1\right|}
\providecommand{\norm}[1]{\left \| #1\right \|}

\newcommand{\real}{\mathbb{R}}
\newcommand{\n}{\mathbb{N}}
\newcommand{\rN}{ {\mathbb{R}^N} }

\numberwithin{equation}{section}
\newtheorem{teo}{Theorem}[section]
\newtheorem{lem}[teo]{Lemma}
\newtheorem{remark}[teo]{Remark}

\newtheorem{defi}[teo]{Definition}
\newtheorem{cor}[teo]{Corollary}
\newtheorem{prop}[teo]{Proposition}

\newtheorem{claim}[teo]{Claim}

\def\calc{\mathcal{C}}
\def\mm{{{\cal M}^-}} 

\begin{document}

\title{\bf\Large A priori estimates and multiplicity for  systems of elliptic PDE with natural gradient growth}
\author[1]{Gabrielle Nornberg\footnote{gabrielle@icmc.usp.br, supported by Fapesp grant 2018/04000-9, São Paulo Research Foundation.}}
\author[2]{Delia Schiera\footnote{d.schiera@uninsubria.it}}
\author[3]{Boyan Sirakov\footnote{bsirakov@mat.puc-rio.br}}
\affil[1]{\small Instituto de Ciências Matemáticas e de Computação, Universidade de São Paulo, Brazil}
\affil[2]{\small Università degli Studi dell'Insubria, Italy}
\affil[3]{\small Pontifícia Universidade Católica do Rio de Janeiro, Brazil}
\date{}

\maketitle

\vspace{-1.5cm}

\begin{center}
\textit{Dedicated to Professor Wei-Ming Ni with admiration}
\end{center}

\smallskip

{\small\noindent{\bf{Abstract.}} We consider fully nonlinear uniformly elliptic cooperative systems with quadratic growth in the gradient, such as
$$
-F_i(x, u_i, Du_i, D^2 u_i)- \langle M_i(x)D u_i, D u_i \rangle =\lambda c_{i1}(x) u_1 + \cdots + \lambda c_{in}(x) u_n +h_i(x),
$$
for $i=1,\cdots,n$, in a bounded $C^{1,1}$ domain $\Omega\subset \mathbb{R}^N$ with Dirichlet boundary conditions; here $n\geq 1$, $\lambda \in\real$, $c_{ij},\, h_i \in L^\infty(\Omega)$,  $c_{ij}\geq 0$,  $M_i$ satisfies $0<\mu_1 I\leq M_i\leq \mu_2 I$, and $F_i$ is an uniformly elliptic Isaacs operator.

We obtain uniform a priori bounds for systems, under a weak coupling hypothesis that seems to be optimal. As an application,
we also establish existence and multiplicity results for these systems, including a branch of solutions which is new even in the scalar case.}

\medskip

{\small\noindent{\bf{Keywords.}} {A priori estimates; Elliptic system; Multiplicity; Existence and nonexistence.}

\medskip

{\small\noindent{\bf{MSC2010.}} {35J47, 35J60, 35J66, 35A01, 35A16, 35P30.}

\section{Introduction}
In this paper we study the following system of fully nonlinear uniformly elliptic equations
\begin{equation}\label{Plambda}\tag{$P_{\lambda}$}
\left\{
\begin{array}{rclc}
	-F_i(x, Du_i, D^2 u_i)-\braket{M_i(x) Du_i, Du_i} & =& \lambda\displaystyle\sum_{j=1}^n  c_{ij}(x) u_j + h_i(x) & \text{ in } \Omega \\
u_1=\cdots=u_n&=&0 & \text{ on } \partial \Omega
\end{array}
\right.
\end{equation}
where $\Omega$ is a bounded $C^{1,1}$ domain in $\mathbb{R}^N$, $\lambda \in \mathbb{R}$, $n,N \ge 1$, $c_{ij}, h_i \in L^\infty(\Omega)$, and $M_i$ is a bounded nondegenerate matrix. Scalar product is denoted with $\braket{\cdot, \cdot}$.
We assume $c_{ij}\ge 0$ in $\Omega$, which means that the system is noncoercive and cooperative when $\lambda>0$. The latter is a parameter which measures the size of the zero order matrix $\mathcal{C}=(c_{ij})_{i,j=1}^n $.

A very particular case, for which our results are new as well, is when each $F_i$ is the Laplacian; $F_i$ can also be a linear operator in nondivergence form
$F_i(x,Du,D^2 u)= \textrm{tr} (A_i(x)D^2 u) + \braket{b_i(x),D u}$, or it can even have a fully nonlinear structure as an Isaacs operator.
We note that nondivergence fully nonlinear equations with natural growth are particularly relevant for applications, since problems with such growth in the gradient are  abundant in control and game theory, and more recently in mean-field problems, where Hamilton-Jacobi-Bellman and Isaacs operators appear  as infinitesimal generators of the underlying stochastic processes. We refer to Section 2 of \cite{BuscaSirakov} for more on applications of this type of systems.

It is notable that the two terms in the left-hand side of  \eqref{Plambda} have the same scaling with respect to dilations, so the second order term is not dominating when we zoom into a given point. This type of gradient dependence is usually named ``natural'' in the literature, and is the object of extensive study. Another important property of \eqref{Plambda} is the invariance of this class of systems with respect to diffeomorphic  changes of variable, in $x$ or $u$.\medskip

We start with a  brief review of the literature for scalar equations ($n=1$).  It is known that the sign of $\lambda$ dramatically influences the solvability and properties of the solution set of \eqref{Plambda}. For the so-called strictly coercive case $\lambda c(x)<<0$,
existence and uniqueness when $F$ is in divergence form goes back to the works \cite{BBGK, BM, BMP1, BMP2, KK78}.
 However, in the case of weakly coercive equations (say, $\lambda=0$) existence and uniqueness can be proved only under a smallness assumption on $c$ and $M$, as was first observed in \cite{FM}.
These works use the weak integral formulation of the equation.

The third author showed in \cite{arma2010} that the same type of existence and uniqueness results can be proved for general coercive equations in nondivergence form, by using techniques based on the maximum principle. In that paper it was also observed, for the first time and with a rather specific example with the Laplacian, that the solution set can be very different in the ``noncoercive" case $\lambda c>0$, and in particular more than one solution may appear. It was also conjectured in that paper that a refined analysis should be doable in order to embrace more general structures.

In the last few years appeared several papers which unveil the complex nature of the solution set for noncoercive equations, in the particular case of the Laplacian -- see \cite{ACJT, CJ, JB, Souplet}. In all these works the crucial a priori bounds for $u$ in the $L^\infty$-norm rely on the fact that the second order operator is the Laplacian, or a divergence form operator.

In \cite{jfa19} we obtained similar results for general operators in nondivergence form, by using different techniques adapted to such operators. In particular, the conjectures in \cite{arma2010} for noncoercive equations were established through a new method of obtaining a priori bounds in the uniform norm.
The method is based on some standard estimates from regularity theory, such as half-Harnack inequalities, and their recent boundary extensions in \cite{B2016}, in addition to a Vázquez strong maximum principle; see also \cite{Sirakov19} for an extensive description of the method.

However, up to our knowledge, nothing was known about systems with natural gradient growth. This is what this work is devoted to, complement and extend the results in \cite{jfa19} to the context of systems of the form \eqref{Plambda}.
We develop a machinery to obtain the crucial a priori bounds for the system \eqref{Plambda} via a nondegeneracy hypothesis on the matrix $\calc(x)$ that seems to be optimal. In combination with these estimates we also exploit a Fredholm theory for fully nonlinear operators with unbounded weight, which turns out to be an important tool in investigating existence and multiplicity of solutions.

It is worth noting that general systems as \eqref{Plambda} do not have variational characterization even if the second order operators $F_i$ are in divergence form, such as the Laplacian; so variational methods do not apply to such systems.
\medskip

The paper is organized as follows. The next section contains the statements of our results. In the preliminary section \ref{Preliminaries} we recall some known results that will be used throughout the text.
Section \ref{section a priori} is devoted to the proofs of the a
priori bounds in the uniform norm for solutions of the noncoercive problem \eqref{Plambda}.
In Section \ref{section multiplicity} we sketch the proof of our existence and multiplicity results, which resemble to the scalar case \cite{jfa19} after some appropriate changes.
Section \ref{section scalar}, in turn, consists of a multiplicity result which is new even for single equations in nondivergence form, see Theorem \ref{th cite introd}. It is based on a version of the anti-maximum principle, proven in section \ref{section eigenvalue} together with some tools involving eigenvalues.

\section{Main Results}\label{main results}

We assume that the matrices $M_i$ satisfy the nondegeneracy condition
\begin{align}\label{M}
\tag{$M$}  \mu_1 I \leq M_i(x)\leq \mu_2 I\;\;\;\mathrm{a.e.\;\; in\;} \Omega
\end{align}
for some $\mu_1, \mu_2>0$, and that $F_i$ in \eqref{Plambda} has the following structure
\begin{equation} \label{SC}
\tag{$SC$} \begin{cases}
F_i(x,0,X)\;\mbox{ is continuous in }\;x\in \overline{\Omega},\\
\mm (X-Y)-b|\vec{p}-\vec{q}| \leq F_i(x,\vec{p},X) - F_i(x,\vec{q},Y)
\leq \mathcal{M}^+ (X-Y)+b|\vec{p}-\vec{q}|
\end{cases}
\end{equation}
for a.e. $x\in\Omega$, where $b\ge0$ and $\mm$, $\mathcal{M}^+$ are the Pucci extremal operators (see the next section) with constants $0<\lambda_P\le\Lambda_P$. For simplicity, the reader may think that each $F_i[u]= F_i(x,Du,D^2 u)$ is in one of the following forms
\begin{equation} \label{models}
\textrm{tr} (A_i(x)D^2 u) + \braket{b_i(x), D u}  \quad\mbox{ or } \quad \mathcal{M}^{\pm}_{\lambda_P,\Lambda_P}(D^2u) \pm b_i (x)|Du|
\end{equation}
where $A_i$ are continuous matrices whose spectrum is in $[\lambda_P,\Lambda_P]$, and $b_i$ are bounded vector functions.  Only at the expense of trivial technicalities we can consider more general operators as in \cite{jfa19}, with zero order terms, and coefficients $b_i, c_{ij}, h_i$ belonging to $L^p$, $p>N$. We prefer to avoid such technicalities here, in order to concentrate on what is new due to the presence of a system rather than a scalar equation.

Solutions of the Dirichlet problem \eqref{Plambda} are understood in the $L^p$-viscosity sense (see Definition \ref{def Lp-viscosity sol} below) and belong to $C(\overline{\Omega})$, so are bounded.
We also use the notion of strong solutions, which are functions in $W^{2,p}_{\mathrm{loc}} (\Omega)$ satisfying the equation almost everywhere.
Strong solutions are viscosity solutions, \cite{KSweakharnack}.
Conversely, it follows from the regularity results in \cite{regularidade} that, if the operator $F_i$  has  property \eqref{Hstrong} below, then viscosity solutions are strong. Hypothesis \eqref{SC} guarantees that the $L^p$-viscosity solutions of \eqref{Plambda} have global  $C^{1,\alpha}$ regularity and estimates, by \cite{regularidade}.

\smallskip

We denote  $F [u]:=(F_1[u_1],\cdots,F_n[u_n]\,)$, $u=(u_1,\cdots,u_n)$, $f=(f_1,\cdots,f_n)$, fix $p>N$, and consider the Dirichlet problem
\begin{align}\label{Fif}
-F[u]=f(x)\;\; \textrm{ in } \Omega, \quad u=0 \;\;\textrm{ on } \partial\Omega.
\end{align}
The model operators in \eqref{models} have the following properties.
\begin{align}\label{ExistUnic M bem definido}
\textrm{For each $f \in L^p(\Omega)^n$, there exists a unique $L^p$-viscosity solution of }  \eqref{Fif}. \tag{$H_1$}
\end{align}
\begin{align} \label{Hstrong} \tag{$H_2$}
\textrm{For each $f\in L^p (\Omega)^n$, any solution $u$ of \eqref{Fif} belongs to $W^{2,p} (\Omega)^n$.}
\end{align}
More generally, operators satisfying \eqref{SC} and convex/concave in the Hessian matrix satisfy \eqref{ExistUnic M bem definido}--\eqref{Hstrong},  by \cite{CCKS, regularidade, Winter}.
We stress that \eqref{Hstrong} above implies $(H_2)$ from \cite{jfa19} in the scalar case, by the proof of the $W^{2,p}$  regularity in \cite{regularidade}.

\smallskip

Since we want to study the way the nature of the solution set changes when we go from negative to positive zero order term (i.e.\ from $\lambda<0$ to $\lambda>0$), we will naturally assume that the problem with $\lambda =0$ has a solution.
\begin{align}\label{H0}
\tag{$H_0$} \textrm{The problem (}P_0\textrm{) has a strong solution }u_0=(u_1^0,\cdots , u_n^0) .
\end{align}

Theorem 1(ii) of \cite{arma2010} ensures \eqref{H0} for instance if $\mu_2h_i$ has small $L^p$-norm for each $i$ (notice that $(P_0)$ is a system of $n$ uncoupled equations, hence Theorem 1 of \cite{arma2010} applies to each of these equations separately).
Examples showing that in general this hypothesis cannot be removed are also found there. The function $u_0$ is the unique $L^p$-viscosity solution of $(P_0)$, by Theorem 1(iii) of \cite{arma2010}.\medskip

We use the following order in the space $E:=C^1(\overline{\Omega})^n$.
\begin{defi}\label{def1.2}
Let $u=(u_1,\cdots,u_n)$, $v=(v_1,\cdots, v_n) \in E$. We denote $u\leq v$ in $\Omega$ to mean $u_i\leq v_i$ in $\Omega$ for all $i=1,\cdots,n$.
Also, we say that $u\ll v$ if, for all $i\in\{1,\dots, n \}$, $u_i<v_i$ in $\Omega$, and for any $x_0\in\partial\Omega$ we have either $u_i(x_0)<v_i(x_0)$, or $u_i(x_0)=v_i(x_0)$ and $\partial_\nu u_i (x_0)<\partial_\nu v_i (x_0)$, where $\vec{\nu}$ is the interior unit normal to $\partial \Omega$.

We also write $u \leq C$ $(\ge C)$ to mean $u_i \leq C$ $($respectively, $\ge C)$ for any $i=1,\cdots,n$.
\end{defi}

\smallskip

As in any study of systems of equations, it is essential to determine the coupling of the system, that is, the way each of the equations influences each of the components of the vector $u$. A {\it fully coupled} system is one which cannot be split into two subsystems such that one of which does not depend on the other. In our context, \eqref{Plambda} would be fully coupled if the matrix $\calc$ is irreducible, in the sense that for each nonempty $I,J\subset\{1,\ldots,n\}$, $I\cap J=\emptyset$, $I\cup J=\{1,\ldots,n\}$ there exist $i\in I$, $j\in J$, such that $c_{ij}(x)\gneqq 0$ in $\Omega$.

Every matrix $\mathcal{C}=(c_{ij})_{i,j=1}^n $ can be written in the block triangular form
\begin{equation}\label{blocktriangularform}
\mathcal{C}(x)=(\calc_{kl} (x))_{\,k, l=1}^{\,{n^\prime}},
\end{equation}
where $1 \le {n^\prime} \le n$, $\calc_{kl}$ are $t_k \times t_l$ matrices, $\sum_{k=1}^{n^\prime} t_k=n$,  $\calc_{kk}$ is irreducible for each $k=1,\ldots, n^\prime$, and $\calc_{kl} \equiv 0$ in $\Omega$, for all $k,l \in \{ 1, \dots, {n^\prime} \}$ with $k <l$. This is easy to achieve by renumbering lines and columns of $\calc$, that is, by changing the order of the equations in \eqref{Plambda} and renumbering the components of $u$. Indeed, if $\calc$ is irreducible, we can take $n^\prime=1$, $\calc_{11}=\calc$; if not, there are two subsets $I,J$ as in the previous paragraph, and we renumber so that $I=\{1,\ldots k\}$ with $k=|I|$, then repeat the same until reaching \eqref{blocktriangularform}. See Section \ref{EA} below, and Section 8 in \cite{BuscaSirakov}.

From now on we assume that $\calc$ in \eqref{Plambda} is in the form \eqref{blocktriangularform}. We will say that $u\ll v$ in some block if there exists some $k\in\{1,\ldots, n^\prime\}$ such that $\widetilde{u}\ll \widetilde{v}$ in $\Omega$, where for any $w\in\mathbb{R}^n$ we denote with $\widetilde{w}$ the vector $(w_{s_{k-1}+1},\cdots , w_{s_k})$, and $s_0=0$, $s_k=\sum_{i=1}^k t_i$.

 The additional assumption that we need to impose, which extends and plays the role of hypothesis $c\gneqq 0$ from the scalar case, is the following.
\begin{align}\label{1x1 blocks nonzero}\tag{$H_3$}
\textrm{In \eqref{blocktriangularform}, there is no $1 \times 1$  block with a zero coefficient, i.e. if $t_k=1$ then $\calc_{kk}\not\equiv0$.}
\end{align}
This hypothesis seems to be optimal for our kind of systems, see Remark \ref{optimal}. To our knowledge, this is the first time such a hypothesis appears in the study of elliptic systems.

We now state our  results.
The first theorem is a uniform estimate for solutions of \eqref{Plambda}, which is both important in itself and instrumental for the existence statements below.

\begin{teo} \label{apriori}
Suppose \eqref{M}, \eqref{SC}, \eqref{1x1 blocks nonzero} hold. Let $\Lambda_1, \Lambda_2$ with $0<\Lambda_1 < \Lambda_2$.
Then every $L^p$-viscosity solution $(u_1, \dots, u_n)$ of \eqref{Plambda} satisfies
\[ \norm{u_i}_{\infty} \le C, \text{ for all } \lambda \in [\Lambda_1, \Lambda_2], \, i=1, \dots, n, \]
where $C$ depends on $n, N, p, \mu_1, \mu_2,  \mathrm{diam}(\Omega), \Lambda_1, \Lambda_2, \norm{b}_{\infty}, \norm{c_{ij}}_{\infty}, \norm{h_i}_{\infty}$, $\lambda_p, \Lambda_p$, and on a lower bound on the measure of the sets where the $c_{ij}$  are positive, for those $i,j$ which determine the irreducibility of the blocks in the form \eqref{blocktriangularform}.
\end{teo}

The next theorems describe the solution set of \eqref{Plambda}.
\begin{teo} \label{th1.1,1.2,1.3}
Assume \eqref{M}, \eqref{SC},  \eqref{H0}, \eqref{ExistUnic M bem definido}, and \eqref{1x1 blocks nonzero}.

1. Then, for $\lambda\leq 0$, the problem \eqref{Plambda} has an $L^p$-viscosity solution $u_\lambda$ that converges to $u_0$ in $E$ as $\lambda\rightarrow 0^-$. Moreover, the set
$
\Sigma = \{ \,(\lambda , u) \in \real \times E\, ; \, u \;\, \textrm{solves \eqref{Plambda}} \,\}
$
possesses an unbounded component $\mathcal{C}^+\subset [0,+\infty]\times E$ such that $\mathcal{C}^+\cap ( \{0\}\times E )=\{u_0\}$.

\smallskip

2. This component is such that:
either it bifurcates from infinity to the right of the axis $\lambda =0$ with the corresponding solutions having a positive part blowing up to infinity in $C (\overline{\Omega})$ as $\lambda\rightarrow 0^+$;
or its projection on the $\lambda$ axis is $[0,+\infty)$.

\smallskip

3. There exists $\bar{\lambda} \in (0,+\infty]$ such that, for every $\lambda\in (0,\bar{\lambda})$, the problem \eqref{Plambda} has at least two $L^p$-viscosity solutions, $u_{\lambda, 1}$ and $u_{\lambda , 2}\,$, satisfying
$u_{\lambda , 1}\rightarrow u_0$ in $E$;
$\max_{\overline{\Omega}} u_{\lambda , 2} \rightarrow +\infty$ as $\lambda\rightarrow 0^+$;
and if $\,\bar{\lambda}<+\infty$, the problem $(P_{\bar{\lambda}})$ has at least one  $L^p$-viscosity solution. The latter is unique if $F(x,\vec{p},X)$ is convex in $(\vec{p},X)$.

\smallskip

4. If \eqref{Hstrong} holds, the solutions $u_\lambda\,$ for $\lambda\leq 0$ are unique among $L^p$-viscosity solutions; whereas the solutions from 3. for $\lambda>0$ are ordered in some block.
If in addition the system is fully coupled, $u_{\lambda, 1} \ll u_{\lambda , 2}$ in the sense of definition \ref{def1.2}, for all $\lambda>0$.
\end{teo}

In the next two theorems, we show that it is possible to obtain a more precise description of the set $\Sigma$, provided we know the sign of $u_0$. For this, we need to extend the hypothesis $c(x)u_0\not\equiv 0$ from the scalar case to the context of the system. The following assumption is  a natural requirement in view of our weak coupling hypothesis \eqref{1x1 blocks nonzero}.
\begin{align}\label{cu0 nonzero}\tag{$H_4$}
\textrm{ $(\calc u_0)_i \not\equiv 0$ for at least one $i \in S_k=\{ s_{k-1} + 1, \dots, s_k \}$, for all $k \in \{1, \dots, n' \}$,}
\end{align}
where $s_0=0$, $s_k=\sum_{i=1}^k t_i$, with $t_i$ and $n'$ coming from \eqref{blocktriangularform}.

Notice that hypothesis \eqref{cu0 nonzero} is consistent with the results obtained for single equations in nondivergence form in \cite{jfa19}.
In the particular case $n'=1$, namely if the system is fully coupled, we recover the assumption $\calc(x)u_0 \not\equiv 0$, as a vector.

\begin{teo} \label{th1.5}
Suppose \eqref{M}, \eqref{SC}, \eqref{H0},  \eqref{ExistUnic M bem definido}, \eqref{Hstrong}, \eqref{1x1 blocks nonzero}, \eqref{cu0 nonzero}, and $u_0 \leq 0$.

Then every nonpositive solution of \eqref{Plambda} with $\lambda>0$ satisfies $u\ll u_0$. Furthermore, for every $\lambda >0$, the problem \eqref{Plambda} has at least two nontrivial strong solutions $ u_{\lambda, 1} \le u_{\lambda , 2}\,$, such that  $u_{\lambda_2 ,1}\ll u_{\lambda_1 ,1} \ll u_0$ if $\,0<\lambda_1<\lambda_2\,$, and
$u_{\lambda , 1} \rightarrow u_0$ in $E$; $\max_{\overline{\Omega}} u_{\lambda , 2} \rightarrow +\infty$ as ${\lambda\rightarrow 0^+}$.
If $F(x,\vec{p},X)$ is convex in $(\vec{p},X)$ then $\max_{\overline{\Omega}}\, u_{\lambda,2}>0$ for all $\lambda>0$.
\end{teo}

\begin{figure}
\centering
\includegraphics[scale=0.31]{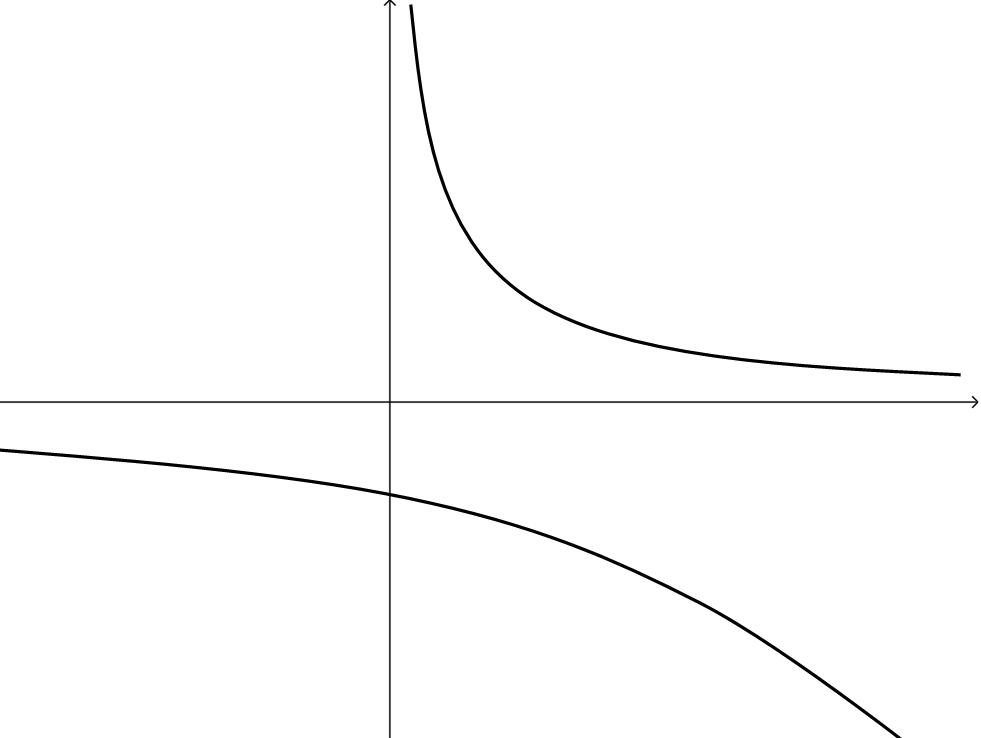}
\caption{Illustration of Theorem \ref{th1.5}.}
\label{figurath15}
\end{figure}

\begin{teo} \label{th1.4}
Suppose \eqref{M}, \eqref{SC}, \eqref{H0}, \eqref{ExistUnic M bem definido}, \eqref{Hstrong}, \eqref{1x1 blocks nonzero}, \eqref{cu0 nonzero}, and $u_0 \geq 0$.
\smallskip

Then every nonnegative solution of \eqref{Plambda} with $\lambda>0$ satisfies $u\gg u_0$. Moreover, there exists  $\bar{\lambda}_1 \in (0,+\infty)$ such that
for every $\lambda\in (0,\bar{\lambda}_1)$, the problem \eqref{Plambda} has at least two nontrivial strong solutions with $ u_{\lambda, 1} \le u_{\lambda , 2}\,$, where $u_0 \ll u_{\lambda_1 ,1}\ll u_{\lambda_2 ,1} $ if $\,0<\lambda_1<\lambda_2\,$,
$u_{\lambda , 1} \rightarrow u_0$ in $E$, and $\max_{\overline{\Omega}} u_{\lambda , 2} \rightarrow +\infty$ as ${\lambda\rightarrow 0^+}$.
The problem $(P_{\bar{\lambda}_1})$ has at least one nonnegative strong solution, which is unique if $F$ is convex in $(\vec{p},X)$; and for $\lambda > \bar{\lambda}_1$, the problem \eqref{Plambda} has no nonnegative solution.

\smallskip

Furthermore, there exists some $\delta >0$ such that, if
$\sup_i \mu_2\|h_i\|_{L^p(\Omega)}\leq \delta$,
with $h\gneqq 0$,
then we have the existence of $\bar{\lambda}_2 > \bar{\lambda}_1$ such that  \eqref{Plambda} has at least two strong solutions for $\lambda> \bar{\lambda}_2$, with $u_{\lambda,1}\ll 0$ in $\Omega$ and $\min_{\overline{\Omega}} u_{\lambda,2}<0$.
The problem  $(P_{\bar{\lambda}_2})$ has at least one nonpositive strong solution, which is unique if $F$ is convex in $(\vec{p},X)$; and
for $\lambda < \bar{\lambda}_2$, the problem \eqref{Plambda} has no nonpositive solution.
\end{teo}

\begin{figure}[!htb]
\centering
\includegraphics[scale=0.43]{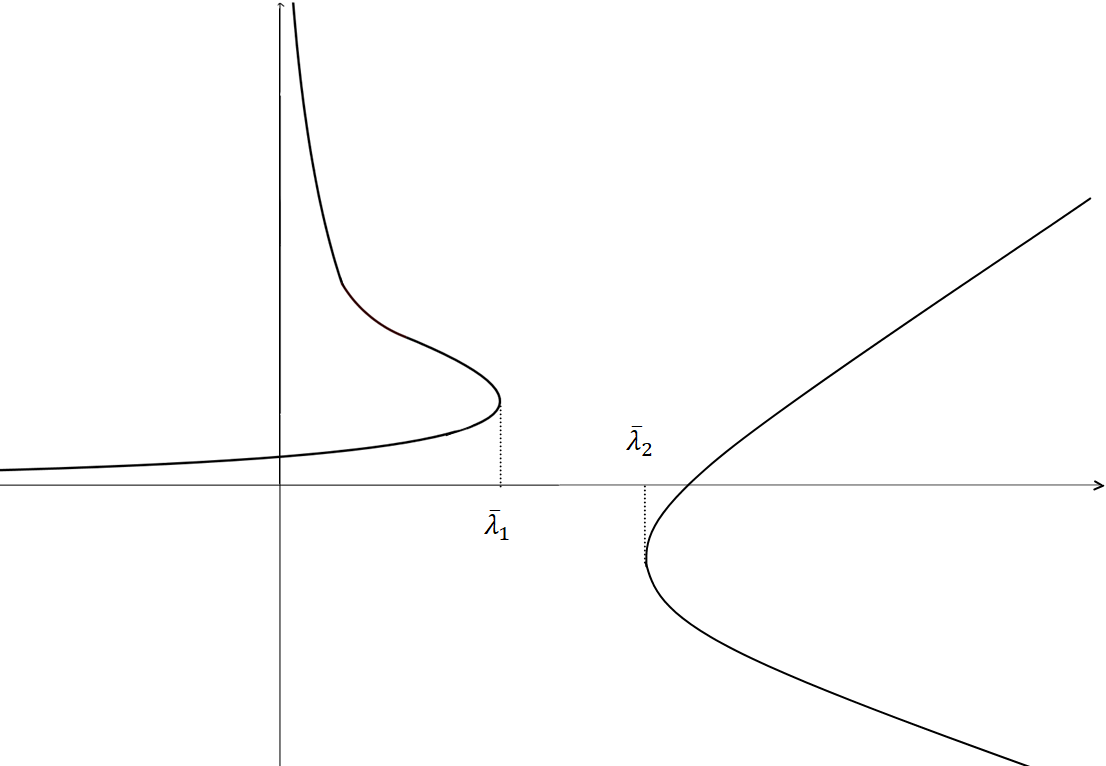}
\caption{Illustration of Theorem \ref{th1.4} for $\mu_2 h\gneqq 0$ small in $L^p$-norm.}
\label{Rotulo2}
\end{figure}

Moreover, as in item \textit{4} of Theorem \ref{th1.1,1.2,1.3}, in theorems \ref{th1.5} and \ref{th1.4} the solutions $u_{\lambda, 1}$, $u_{\lambda , 2}$ are ordered in at least one block; and $u_{\lambda, 1} \ll u_{\lambda , 2}$ in the sense of definition \ref{def1.2}, for all $\lambda>0$ if \eqref{Plambda} is fully coupled, see Claim \ref{ulambda,1 leq ulambda,2 th1.3}.

We remark that the hypotheses $u_0\le 0$, resp $u_0\ge0$, of the above theorems are implied for instance by $h\le 0$, resp $h\ge0$. See Remark 6.25 of \cite{jfa19} for a proof.

We stress that theorems \ref{apriori}--\ref{th1.4} are new even for systems involving the Laplacian operator.
Moreover, the second part in Theorem \ref{th1.4} is new even for a single equation, in the context of nondivergence form operators.


\section{Preliminaries}\label{Preliminaries}

In this section we briefly recall some definitions and previous results which we use in the sequel.
More comments can be found in the preliminary section of \cite{jfa19}.

Let $F_i\,(x,\vec{p},X):\Omega\times\rN\times\mathbb{S}^N\rightarrow\real$ be a measurable function satisfying \eqref{SC}, where
$$
\mathcal{M}^+(X):=\sup_{\lambda_P I\leq A\leq \Lambda_P I} \mathrm{tr} (AX)\,,\quad \mathcal{M}^-(X):=\inf_{\lambda_P I\leq A\leq \Lambda_P I} \mathrm{tr} (AX)
$$
are the Pucci's extremal operators with constants $0<\lambda_P\leq \Lambda_P$.
See, for example, \cite{CafCab} for their properties.
Also, denote $\mathcal{L}^\pm [u]:=\mathcal{M}^\pm (D^2 u)\pm b|Du|$, for $b\ge0$.

\begin{defi}\label{def Lp-viscosity sol}
Let $f\in L^p_{\textrm{loc}}(\Omega)^n$.
We say that $u\in C(\Omega)$ is an $L^p$-viscosity subsolution $($respectively, supersolution$)$ of the system $F[u]=f(x)$ in $\Omega$ if, for each $i\in \{1,\cdots , n\}$, whenever $\phi\in  W^{2,p}_{\mathrm{loc}}(\Omega)$, $\varepsilon>0$ and $\mathcal{O}\subset\Omega$ open are such that
\begin{align*}
F_i(x,u_i(x),D\phi(x),D^2\phi (x))-f_i(x)  \leq -\varepsilon\;\;
( F_i(x,u_i(x),D\phi(x),D^2\phi (x))-f_i(x)  \geq \varepsilon )
\end{align*}
for a.e. $x\in\mathcal{O}$, then $u_i-\phi$ cannot have a local maximum $($minimum$)$ in $\mathcal{O}$.
\end{defi}

If both $F_i$ and $f_i$ are continuous in $x$, for all $i=1,\cdots , n$, we can use the more usual notion of  \textit{$C$-viscosity} sub and supersolutions -- see \cite{user}.

On the other side, a \textit{strong} sub or supersolution belongs to $W^{2,p}_{\mathrm{loc}}(\Omega)^n$ and satisfies the inequality at almost every point. As we already mentioned, this is intrinsically connected to the notion of  $L^p$-viscosity solution; more precisely we have the following fact.
\begin{prop}\label{Lpiffstrong.quad}
Let $F_i$ satisfy \eqref{SC} and $f_i\in L^p(\Omega)$, $\mu\ge0$.
Then, $u_i\in W^{2,p}_{\mathrm{loc}}(\Omega)$ is a strong subsolution $($supersolution$)$ of $F_i[u_i]+\mu|D u_i|^2=f_i$ in $\Omega$ if and only if it is an $L^p$-viscosity subsolution $($supersolution$)$ of this equation.
\end{prop}

See Theorem 3.1 and Proposition 9.1 in \cite{KSweakharnack} for a proof.
For scalar equations it is also well known that the pointwise maximum of subsolutions, or supremum over any set $($if this supremum is locally bounded$)$, is still a subsolution, see \cite{KoikePerron}.

The next proposition follows from Theorem 4 in \cite{arma2010} or Proposition 9.4 in \cite{KSweakharnack}.

\begin{prop} {$($Stability$)$} \label{Lpquad}
Let $F$, $F_k$ be scalar operators satisfying \eqref{SC}, $p> N$, $f, \, f_k\in L^p(\Omega)$, $u_k\in C(\Omega)$ an $L^p$-viscosity subsolution $($supersolution$)$ of
$$
F_k(x,u_k,Du_k,D^2u_k)+\langle M(x) Du_k,Du_k\rangle \geq(\leq) f_k(x) \;\;\textrm{in} \;\;\Omega\, , \;\textrm{ for all } k\in \n .
$$
Suppose $u_k\rightarrow u$ in $L_{\mathrm{loc}}^\infty  (\Omega)$ as $k\rightarrow \infty$ and, for each $B\subset\subset \Omega$ and $\varphi\in W^{2,p}(B)$, if we set
\begin{align*}
g_k(x):=F_k(x,u_k,D\varphi,D^2\varphi) \rangle- f_k(x) \, , \;
g(x):=F(x,u,D\varphi,D^2\varphi)-f(x)
\end{align*}
we have $\| (g_k-g)^+\|_{L^p(B)}$ $(\| (g_k-g)^-\|_{L^p(B)}) \rightarrow 0$ as $k\rightarrow \infty$. Then $u$ is an $L^p$-viscosity subsolution $($supersolution$)$ of\;
$
F(x,u,Du,D^2u)+\langle M(x)Du,Du\rangle  \geq(\leq) f(x) \,\textrm{ in }\, \Omega\, .
$
\end{prop}

The following result follows from Lemma 2.3 in \cite{arma2010}, see also the appendix of \cite{jfa19}.

\begin{lem}{$($Exponential change$)$}\label{lemma2.3arma}
Let $p> N$ and $u\in C(\Omega)$. For $m>0$ set
$mv=e^{mu}-1$ and $ mw=1-e^{-mu}$.
Then the following inequalities hold in the $L^p$-viscosity sense
\begin{align*}
\mathcal{M}^\pm (D^2 u)+m\lambda_P |Du|^2 &\leq \frac{\mathcal{M}^\pm (D^2 v)}{1+mv} \leq \mathcal{M}^\pm (D^2 u)+m\Lambda_P |Du|^2 , \\
\mathcal{M}^\pm (D^2 u)-m\Lambda_P |Du|^2 &\leq \frac{\mathcal{M}^\pm (D^2 w)}{1-mw} \leq \mathcal{M}^\pm (D^2 u)-m\lambda_P |Du|^2 .
\end{align*}
\end{lem}\vspace{0.1cm}

\vspace{0.05 cm}

The following scalar estimates will play a pivotal role in our proofs. The first one is a global variant of the Local Maximum Principle (LMP);
 see \cite{tese, B2016} for a proof.

\begin{teo}[GLMP]\label{LMP}
Let $u$ be a locally bounded $L^p$-viscosity subsolution of
\begin{align*}
\left\{
\begin{array}{rclcl}
\mathcal{L}^+ (D^2 u)+\nu (x) u &\geq & -f(x) &\mbox{in} &\Omega \\
u &\leq & 0 &\mbox{on} & \partial\Omega
\end{array}
\right.
\end{align*}
with $f\in L^p (\Omega)$, $\nu\in L^{p_1} (\Omega)$, for some $p, p_1>N$.
Then, for each $r>0$,
\begin{align*}
 \sup_{\Omega} u^+ \leq C \left( \left(  \int_{\Omega} (u^+)^r \right)^{1/r} + \|f^+\|_{L^p(\Omega)} \right),
\end{align*}
where $C$ depends only on $N,\,p,\, p_1,\,\lambda, \,\Lambda,\, r,\, b$, and $\,\|\nu\|_{L^{p_1}(\Omega)}$.
\end{teo}

We recall the following two global scalar versions of the quantitative strong maximum principle (QSMP) and the weak Harnack inequality (WHI), which follow from theorems 1.1 and 1.2 in \cite{B2016}.
Denote $d=d(x)=dist(x, \partial \Omega)$.

\begin{teo}[GQSMP]\label{QSMP}
Assume  $u$ is an $L^p$ viscosity supersolution of $\mathcal{L}^- [u] - gu \le f$, $u \ge 0$ in $\Omega$, and let $f, g \in L^p(\Omega)$, $p>n$. Then there exist constants $\varepsilon, c, C>0$ depending on $n, \lambda, \Lambda, b, p, $ and $\norm{g}_p$ such that
\[ \inf_{\Omega} \frac{u}{d} \ge c \left( \int_{\Omega} (f^-)^{\varepsilon} \right)^{1/\varepsilon} - C \norm{f^+}_p . \]
\end{teo}

\begin{teo}[GWHI]\label{WHI}
Suppose $g, f \in L^p$, $p>n$. Assume  $u$ is an $L^p$ viscosity supersolution of $\mathcal{L}^-[u] - gu \le f$, $ u\ge 0$ in $\Omega$. Then there exist constants $\varepsilon, c, C>0$ depending on $n, \lambda, \Lambda, b, p$ and $\norm{g}_p$ such that
\[ \inf_{\Omega} \frac{u}{d} \ge c \left( \int_{\Omega} \left( \frac{u}{d} \right)^{\varepsilon} \right)^{1/\varepsilon} - C \norm{f^+}_p . \]
\end{teo}

In \cite{B2016}, theorems \ref{QSMP} and \ref{WHI} are proved for $g\equiv 0$, but exactly the same proofs there work for any $g\ge0$. Moreover, since the function $u$ has a sign, $g^-u\geq 0$ and they are also valid for nonproper operators.
Theorem \ref{WHI} implies, in particular, the strong maximum principle (SMP) for single equations when $f=0$, i.e. for $\Omega\in C^{1,1}$ and $u$ an $L^p$-viscosity solution of $\mathcal{L}^-[u]-gu\leq 0$, $u\geq 0$ in $\Omega$, where $g\in L^p(\Omega)$, we have either $u\equiv 0$ in $\Omega$ or $u>0$ in $\Omega$; in the latter case, if $u(x_0)=0$ at $x_0\in\partial\Omega$, then $\partial_\nu u (x_0)>0$, by Hopf lemma.
We are going to refer to these simply as SMP and Hopf throughout the text.

\section{A priori estimates for systems}\label{section a priori}

This section contains the proof of Theorem \ref{apriori}, that is,  we establish uniform a priori bounds for the system \eqref{Plambda}.
We will develop ideas in \cite{note, Sirakov19}.

For simplicity, we carry over the proofs in the model case $n=2$. We just refer to the differences from the general case when  needed.

\subsection{Estimates from below}\label{EB}

The first step to obtain a priori estimates, as in \cite[Section 5]{jfa19}, is to prove that any $L^p$-viscosity supersolution of \eqref{Plambda}  is uniformly bounded from below.

\begin{teo}\label{estimatesbelow}
Suppose \eqref{SC} and let $\Lambda_2>0$. Then every $L^p$-viscosity supersolution $(u_1, \dots, u_n)$ of \eqref{Plambda} satisfies
\[ \norm{u_i^-}_{\infty} \le C_1, \quad \text{ for all } \lambda \in [0, \Lambda_2], \;\; i=1, \dots, n, \]
where $C_1$ depends only on $n, N, p, \mu_1, \Omega, \Lambda_2, \norm{b}_{\infty}, \norm{c_{ij}}_{\infty}, \norm{h_i^-}_{\infty},  \lambda_p, \Lambda_p$.
\end{teo}

\smallskip

\noindent{\it Proof.} First we take $U_i=u_i^-$ and we make the following exponential change
\[ w_i=\frac{1- e^{-mU_i}}{m}, \;\;\;i=1,2, \;\;\text{ with } m=\frac{\mu_1}{\Lambda_p}. \]
By Lemma \ref{lemma2.3arma} we know that $(w_1, w_2)$ satisfies
\begin{align*}
-\mathcal{L}_1^+[w_i]  \le\, &  \frac{\lambda}{m} c_{i1}(x)  \abs{ \ln(1-m w_1)}(1-m w_i) \\
	&+  \frac{\lambda}{m} c_{i2}(x)  \abs{ \ln(1-m w_2)}(1-m w_i)   + h^-(x) \;\; \textrm{ in } \Omega
\end{align*}
with $w_i=0$ on $\partial \Omega$, where $\mathcal{L}_1^+[w] = \mathcal{L}^+[w] - m h^- (x) w$ and $h^-=\max \{h_1^-,h_2^-\}$.

Now we consider
\begin{equation}\label{eqn:max}
\begin{cases}
-\mathcal{L}_1^+[w] \le h^-(x)+ \frac{2\lambda}{m} c(x)  \abs{ \ln(1-m w)}(1-m w) & \text{ in } \Omega \\[1ex]
\;\;w=0 & \text{ on } \partial \Omega
\end{cases}
\end{equation}
where $c=\max_{i, j} \{ c_{ij} \}$.
Notice that $w=\max \{w_1, w_2 \}$ satisfies \eqref{eqn:max}.
Define
\[ \bar w= \sup \mathcal{A}, \text{ where } \mathcal{A}:= \{ w \colon w \text{ is an $L^p$-visc.~solution of \eqref{eqn:max}; } 0 \le w < 1/m \text{ in } \Omega \}. \]
As a supremum of subsolutions, $\bar w$ is a subsolution of \eqref{eqn:max}.

Next we proceed as in \cite[Proposition 5.2]{jfa19} to prove that $\bar w \ne \frac{1}{m}$.
Indeed,
\[ \bar w(x) \le C \norm{f^+}_p \text{dist}(x, \partial \Omega) \to 0 \quad \text{ as } x \to \partial \Omega \]
where
\[ f(x)=h^-(x) + \frac{\lambda}{m} c(x) \abs{ \ln (1-m\bar w)} (1-m \bar w). \]
Assume by contradiction that there exists a sequence of supersolutions $(u_1^k, u_2^k)$ of \eqref{Plambda} with unbounded negative parts, namely there exists a subsequence such that
\[
(u_1^k)^-(x_k)=\norm{ (u_1^k)^-}_{\infty} \to \infty, \, x_k \in \bar \Omega, \, x_k \to x_0 \in \bar \Omega
\]
with $x_k \in \Omega$ for large $k$ since $u_1^k \ge 0$ on $\partial \Omega$.
One has
\[ w_1^k(x_k)=\frac{1}{m} \{ 1 - e^{-m (u_1^k)^-(x_k)} \} \to \frac{1}{m}. \]
Take $w^k=\max \{ w_1^k, w_2^k \} < 1/m$. Then,
\[ w^k (x_k) \to \frac{1}{m} \]
and $w^k \in \mathcal{A}$. In particular, for every $\varepsilon >0$ there exists $k_0$ such that
\[ \frac{1}{m} \ge \bar w(x_k) \ge w^k(x_k) \ge \frac{1}{m} - \varepsilon, \, \text{ for all } k \ge k_0 \]
thus
\[ \bar w(x_0) \ge \liminf_{x_k \to x_0} \bar w (x_k)=\lim_{k \to \infty} \bar w (x_k)=\frac{1}{m}. \]
As a consequence, $x_0 \in \Omega$ and $\bar w (x_0)=\frac{1}{m}$.
Then we reach a contradiction as in \cite[Proposition 5.2]{jfa19}, by applying a nonlinear version of the strong maximum principle \cite[Lemma 5.3]{jfa19}.

\subsection{Estimates from above} \label{EA}

First we recall that the matrix $\mathcal{C}=(c_{ij})_{i,j=1}^n$ is said to be irreducible -- equivalently we say that the system \eqref{Plambda} is fully coupled for $\lambda>0$ -- if for any nonempty sets $I, J \subset \{1, \cdots , n\}$ such that $I \cap J = \emptyset$ and $I \cup J = \{1, \cdots , n\}$, there exist $i_0 \in I$ and $j_0 \in J$ for which
\begin{align}\label{ci0j0 >0}
\mathrm{meas}\{x \in \Omega ; \; c_{i_0 j_0 } (x) > 0\} > 0.
\end{align}
This means that the system cannot be split into two subsystems in which one of them does not depend on the other. For instance, if $n=2$, it says that $c_{12}\gneqq 0$ and $c_{21}\gneqq 0$ in $\Omega$.
Of course if both $c_{12}$ and $c_{21}$ are identically zero, then we already know multiplicity from \cite{jfa19}, as soon as $c_{11}\gneqq 0$ and $c_{22}\gneqq 0$.

For simplicity, when \eqref{ci0j0 >0}  holds we write $c_{i_0 j_0}\gneqq 0$ in $\Omega$.
We can fix $\rho > 0$ such that the sets $\{x \in B_R ; \; c_{i_0 j_0} (x) \ge \rho \}$ have positive measures. Let $\omega > 0$ be a lower bound for these measures.

Then we recall our main result concerning a priori estimates for systems.
\begin{teo}\label{estimatesabove}
Suppose \eqref{SC} holds and let $\Lambda_1, \Lambda_2$ with $0<\Lambda_1 < \Lambda_2$.
Assume further that $\mathcal{C}(x)=(c_{ij})_{i,j=1}^n$ is in the block triangular form \eqref{blocktriangularform}, and that \eqref{cu0 nonzero} holds, namely $\mathcal{C}$ has no $1 \times 1$ diagonal blocks with a zero coefficient.
Then every $L^p$-viscosity solution $(u_1, \dots, u_n)$ of \eqref{Plambda} satisfies
\[ \norm{u_i}_{\infty} \le C, \text{ for all } \lambda \in [\Lambda_1, \Lambda_2], \, i=1, \dots, n, \]
where $C$ depends on $n, N, p, \mu_1, \mu_2, \mathrm{diam}\Omega, \Lambda_1, \Lambda_2, \norm{b}_{\infty}, \norm{c_{ij}}_{\infty}, \norm{h_i}_{\infty}$, $\lambda_p, \Lambda_p$, and $\omega$.
\end{teo}

\begin{remark}\label{optimal}
Notice that if $\,\mathcal{C}(x)$ is in the form \eqref{blocktriangularform} and has a $1 \times 1$ diagonal block with a zero coefficient, then there is no chance of getting a priori bounds for \eqref{Plambda}, in general. Indeed, say that block is in the $i_0$-th line. Even if we could prove that all preceding functions $u_1,\ldots,u_{i_0-1}$ are uniformly bounded, then $u_{i_0}$ solves a scalar equation without a zero-order term. Specifically, $u_{i_0}$ solves an equation like $(P_0)$, but with $h_{i_0}$ replaced by $h_{i_0}+\lambda\sum_{j=1}^{i_0-1} c_{i_0j}u_j$; however, as we recalled after \eqref{H0} such an equation admits in general a priori bounds only if $h_{i_0}$ is small, while resonance phenomena may appear otherwise, see \cite{siam2010} and \cite{arma2010}.

See also section~\ref{section scalar} for a two parameter dependence in the problem \eqref{Plambda}, obtained for a large parameter $\lambda$ but a small $h$.
\end{remark}

\begin{remark}
Clearly, if \eqref{Plambda} is fully coupled then it satisfies the hypotheses of Theorem~\ref{estimatesabove}, just take ${n^\prime}=1$. The other extreme is a diagonal matrix such that $c_{kk} \gneqq 0$ for any $k$, by choosing ${n^\prime}=n$, which corresponds to $n$ independent scalar equations with positive zero-order term coefficients, and Theorem \ref{estimatesabove} reduces to \cite[Theorem 2.1]{jfa19}.
\end{remark}

Now we prove Theorem \ref{estimatesabove}.
As a first step, we assume that \eqref{Plambda} is fully coupled. Again, in order to avoid cumbersome notation, we assume $n=2$, and we point out how to adapt the proof for $n \ge 2$ when necessary.

By Theorem \ref{estimatesbelow}, solutions are bounded from below by a uniform constant $C_1$.
Fix $\delta>0$.
Notice that $v_i:=u_i+C_1+\delta$, is a nonnegative viscosity solution of

\begin{align*}
\mathcal{M}^-(D^2v_i)-b\abs{Dv_i} \le -\lambda c_{i1}(x) v_1 - \lambda c_{i2}(x) v_2 - \mu_1 \abs{Dv_i}^2 + \tilde h_i(x)\;\;\textrm{ in } \Omega,
\end{align*}
where $\tilde h_i=h_i^- + \Lambda_2 \{c_{i1}+c_{i2}\} (C_1+\delta)$. Thus, by Lemma \ref{lemma2.3arma}, the functions
\[ w_i:=\frac{1}{m_1} \{ e^{m_1 v_i}- 1 \}, \;\;\; i=1,2, \]
where $m_1=\frac{\mu_1}{\Lambda_p}$, form a nonnegative viscosity supersolution of
\[
\mathcal{L}_{i}^-[w_i] \le f_i(x)  \text{ in } \Omega , \quad i=1,2
\]
with $\mathcal{L}_i^-[w]=\mathcal{M}^- (D^2w) - b\abs{D w} - m_1 \tilde h_i(x) w$ and $f_i(x)=-\frac{\lambda}{m_1}c_{i1}(x)(1+m_1 w_i) \ln (1+m_1 w_1) -\frac{\lambda}{m_1}c_{i2}(x)(1+m_1 w_i) \ln (1+m_1 w_2) +\tilde h_i(x)$.
Let
\[ I_1 = \inf_{\Omega} \frac{w_1}{d} , \quad I_2 = \inf_{\Omega} \frac{w_2}{d}. \]
Since $f_i^+ \in L^p(\Omega)$ (see the proof of Theorem 5.1 in \cite{jfa19}), we can apply Theorem \ref{QSMP} to obtain suitable constants such that
\begin{align*}
I_1 & \ge c_0 \left( \int_{\Omega} (f_1^-)^{\varepsilon} \right)^{1/\varepsilon} - C_0 \norm{f_1^+}_{L^p}= c_0 \Big (\int_{\Omega} \Big \{ \Big ( \frac{\lambda}{m_1} c_{11}(x) (1 + m_1 w_1) \ln  (1 + m_1 w_1)\\
& \begin{multlined}
	 \;\;\; +\frac{\lambda}{m_1} c_{12}(x) (1 + m_1 w_1) \ln  (1 + m_1 w_2) - \tilde h_1(x) \Big )^+ \Big \}^{\varepsilon} \Big)^{1/\varepsilon} - C
\end{multlined}\\
& \begin{multlined}
	\ge c_0 \inf_{\Omega} \frac{w_1}{d} \Big ( \int_{\Omega} \Big ( \Big ( \lambda c_{11}(x) \frac{1+m_1w_1}{m_1w_1} d \ln (1 + m_1w_1) \\
	+ \lambda c_{12}(x) \frac{1+m_1w_1}{m_1w_1} d \ln (1 + m_1w_2) - \tilde h_1(x) \frac{1+m_1w_1}{w_1}d \Big)^+ \Big)^{\varepsilon} \Big)^{1/\varepsilon} - C
\end{multlined}\\
& \begin{multlined}
	\ge c_0 I_1 \Big( \int_{\Omega} \Big \{ \Big( \lambda c_{11}(x) d  \ln (1+I_1 m_1 d) + \lambda c_{12}(x) d \ln (1+I_2 m_1 d) \\
		 - m_1 \tilde h_1(x) d \Big)^+ \Big \}^{\varepsilon} \Big)^{1/\varepsilon} - C
    \end{multlined}
\end{align*}
Therefore
\begin{multline}\label{estimatesI1}
 I_1 \Big \{ c_0  \Big ( \int_{\Omega} d^{\varepsilon}  \Big(  \Big( \lambda c_{11}(x) \ln(1+I_1m_1 d) +  \lambda c_{12}(x) \ln(1+I_2m_1 d)\\
 - m_1 \tilde h_1(x)  \Big)^+  \Big)^{\varepsilon}  \Big)^{1/\varepsilon} - 1  \Big \} \le C
 \end{multline}
and analogously
\begin{multline}\label{estimatesI2}
 I_2 \Big \{ c_0  \Big ( \int_{\Omega} d^{\varepsilon}  \Big(  \Big( \lambda c_{21}(x) \ln(1+I_1m_1 d) +  \lambda c_{22}(x) \ln(1+I_2m_1 d)\\
 - m_1 \tilde h_2(x)  \Big)^+  \Big)^{\varepsilon}  \Big)^{1/\varepsilon} - 1  \Big \} \le C.
\end{multline}
We prove in the sequel that both $I_1$ and $I_2$ are bounded from above.
By full coupling, $c_{12} \gneqq 0$ and $c_{21} \gneqq 0$. Since $I_1 \ge \frac{1}{m_1}\{e^{m_1\delta} - 1\} >0$, \eqref{estimatesI1} implies
\[ \int_{\Omega} d^{\varepsilon} \Big ( \Big (\lambda c_{11}(x) \ln (1+I_1m_1 d) + \lambda c_{12}(x) \ln (1+I_2m_1 d) - m_1 \tilde h_1(x) \Big)^+ \Big)^{\varepsilon} \le C. \]
In particular,
\[ \int_{\Omega} d^{\varepsilon} \Big ( \Big ( \lambda c_{12}(x) \ln (1+I_2m_1 d) - m_1 \tilde h_1(x) \Big)^+ \Big)^{\varepsilon} \le C, \]
and analogously by \eqref{estimatesI2},
\[ \int_{\Omega} d^{\varepsilon} \Big ( \Big ( \lambda c_{21}(x) \ln (1+I_1m_1 d) - m_1 \tilde h_2(x) \Big)^+ \Big)^{\varepsilon} \le C, \]
whence $I_1, I_2 \le C$ as in \cite[p.1829]{jfa19}.
In the general case $n \ge 2$, we just observe that by full coupling for any fixed $k=1, \dots, n$ there exists an index $j =1, \dots, n $, $j \ne k$, such that $c_{jk} \gneqq 0$. Thus, exploiting the $j$-th equation we get
\[ \int_{\Omega} d^{\varepsilon} \Big ( \Big ( \lambda c_{jk}(x) \ln (1+I_k m_1 d) - m_1 \tilde h_j(x) \Big)^+ \Big)^{\varepsilon} \le C, \]
and $I_k$ turns out to be bounded, for all $k=1,\cdots, n$.

Let us now turn back to the model case $n=2$.
By Theorem \ref{WHI} and $I_1 \le C$ we find constants such that
\begin{equation}\label{BWH_1}
\left( \int_{\Omega} (w_1)^{\varepsilon_1} \right)^{1/\varepsilon_1} \le \textrm{diam}\Omega \left( \int_{\Omega} \left( \frac{w_1}{d} \right)^{\varepsilon_1} \right)^{1/\varepsilon_1} \le C_0\, \{ I_1 + \norm{f_1^+}_p \} \le C.
\end{equation}
Similarly, using $I_2 \le C$ we obtain
\begin{equation}\label{BWH_2}
\left( \int_{\Omega} (w_2)^{\varepsilon_2} \right)^{1/\varepsilon_2} \le C.
\end{equation}

Set
\[ z_i = \frac{1}{m_2} \{ e^{m_2 u_i} - 1\},  \; i=1,2,\]
where $m_2=\frac{\mu_2}{\lambda_p}$. Since
\begin{align*}
\mathcal{M}^+ & (D^2 z_i) + b \abs{D z_i} +  \frac{\lambda}{m_2 z_i} c_{i1}(x)(1+m_2 z_i) \ln (1+m_2 z_1)  z_i  \\
&+ \frac{\lambda}{m_2 z_i} c_{i2}(x)(1+m_2 z_i) \ln (1+m_2 z_2)  z_i\ge -h_i^+(x) (1+m_2z_i) \;\;  \textrm{ in } \Omega,
\end{align*}
with $z_i=0$ on $ \partial \Omega$, then $z:= \max \{ z_1, z_2 \}$ satisfies the following problem
\begin{equation}\label{BLMPsystem}
\left\{
\begin{array}{rclcc}
\mathcal{M}^+ (D^2 z) + b \abs{D z} + \nu(x) z &\ge & -h^+(x) & \mbox{in}  & \Omega \\
z&=&0 & \mbox{on}& \partial \Omega,
\end{array}
\right.
\end{equation}
where
\[ \nu(x)= \frac{2\lambda}{m_2 z} c(x)(1+m_2 z) \ln (1+m_2 z) + m_2 h^+ (x),  \]
$c(x)=\max_{i,j} \{ c_{ij} (x) \} \chi_{\{z >0 \}}$, and $h^+=\max \{ h_1^+,h_2^+\}$.
Notice that
\[ z_i=\frac{1}{m_2} \{ (1+m_1 w_i)^{\frac{m_2}{m_1}} e^{-m_2(C_1+\delta)} - 1 \}. \]
Moreover, for any $s$ there exists $C_s$ such that
\[ \abs{\nu} \le C_s c(x) (1+ \abs{z}^s).  \]

Set $\varepsilon = \min \{ \varepsilon_1, \varepsilon_2 \}$. If we take $s=\varepsilon \frac{m_1}{m_2} \frac{p-N}{p(p+N)}$ and $p_1=\frac{p+N}{2}$, then, by H{\"o}lder, given $\frac{1}{p_1}= \frac{1}{p} + \frac{1}{p_2}$, we obtain
\begin{multline*}
\norm{c \abs{z}^s}_{p_1} \le \norm{c}_{p} \norm{ \abs{z}^{s}}_{p_2}
\le  \norm{c}_{p} \norm{ \abs{z_1}^{s}}_{p_2} +  \norm{c}_{p} \norm{ \abs{z_2}^{s}}_{p_2} \\
= \norm{c}_p \left( \int_{\Omega} \abs{z_1}^{\varepsilon \frac{m_1}{m_2}} \right)^{\frac{p-N}{p(p+N)}} + \norm{c}_p \left( \int_{\Omega} \abs{z_2}^{\varepsilon \frac{m_1}{m_2}} \right)^{\frac{p-N}{p(p+N)}}.
\end{multline*}
Recall that both $I_1$ and $I_2$ are bounded from above, and both \eqref{BWH_1} and \eqref{BWH_2} are satisfied. Then
\[ \norm{\nu}_{p_1} \le C \norm{c}_p + \norm{c \abs{z}^s}_{p_1} +m_2 \|h^+\|_{p} \le C. \]
Thus we have, by Theorem \ref{LMP} applied to \eqref{BLMPsystem},
\[ \sup_{\Omega} z^+ \le C \left \{ \left( \int_{{\Omega}} \abs{z_1}^{\varepsilon_1 \frac{m_1}{m_2}} \right)^{\frac{m_2}{\varepsilon_1 m_1}} + \left( \int_{{\Omega}} \abs{z_2}^{\varepsilon_2 \frac{m_1}{m_2}} \right)^{\frac{m_2}{\varepsilon_2 m_1}} + \norm{h^+}_p  \right \}  \le C. \]
Hence, $u_1^+$ and $u_2^+$ are uniformly bounded in $\Omega$.
This proves that Theorem \ref{estimatesabove} holds for any fully coupled system.

Next, take a system whose matrix is in the block triangular form \eqref{blocktriangularform}, with no $1 \times 1$ zero diagonal blocks. Consider the first $t_1$ equations. They are either a fully coupled system (if $t_1>1$), or a scalar equation with a nonvanishing zero order coefficient (if $t_1=~1$). Hence, by the above and \cite[Theorem 5.1]{jfa19} we conclude that $u_1, \dots, u_{t_1}$ are uniformly bounded. We can now consider these $t_1$ functions as being part of the $h$-terms in the next $t_2$ equations, which in turn become a fully coupled system (if $t_2>1$) or a scalar equation with a positive zero order coefficient (if $t_2=1$). The reasoning iterates, and one proves uniform bounds for $u_1, \dots, u_n$.

\section{Multiplicity results for systems}\label{section multiplicity}

In this section we extend to systems the arguments  in \cite{jfa19}.
Our goal is to point out the main differences that come from the nature of the system, and refer to \cite{jfa19} for further details and references.

Throughout this section, $\langle M(x)D u, D u \rangle$ will be the shorthand notation for the vector with entries $\langle M_i(x)D u_i, D u_i \rangle$, $i=1,\cdots , n$.
We set $E:=C^1(\overline{\Omega})^n$, the Banach space with the norm $\|u\|_E=\max_{1\leq i \leq n}\|u_i\|_{C^1 (\overline{\Omega})}$, where $u=(u_1,\cdots, u_n)$.

\smallskip

We start with some auxiliary results.

\begin{defi} \label{def2.1}
An $L^p$-viscosity subsolution $\xi\in E$ $\mathrm{(}$respectively, supersolution $\eta)$ of \eqref{Plambda} is said to be strict if every $L^p$-viscosity supersolution $($subsolution$)$ $u\in E$ of \eqref{Plambda} such that $\xi\leq u$ $(u \leq\eta)$ in $\Omega$, also satisfies $\xi\ll u$ $(u\ll \eta)$ in $\Omega$.
\end{defi}

Under hypothesis \eqref{ExistUnic M bem definido}, we define the operator $\mathcal{T}_\lambda :\, E\rightarrow E$ that takes $u=(u_1,\cdots,u_n)$ into $\mathcal{T}_\lambda u=U=(U_1, \cdots,U_n)$\,, the unique $L^p$-viscosity solution of the problem
\begin{align}\label{T_u} \tag{$\mathcal{T}_\lambda^u$}
-F[U]= \lambda \calc(x)u+\langle M(x)D u, D u \rangle +h(x) \mbox{ in }  \Omega, \;\;\;
U = 0 \mbox{ on }  \partial\Omega,
\end{align}
for any $\lambda\in \real$, where $h=(h_1,\cdots , h_n)$.

\begin{teo}\label{th2.1}
Suppose \eqref{SC},  and \eqref{ExistUnic M bem definido}.
Let $\xi =\max_{1\leq i \leq \kappa} \,\xi_i \,,\, \eta = \min_{1\leq j \leq \iota} \,\eta_j \,$, where $\xi_i \, , \,\eta_j \in W^{2,p} (\Omega)^n$ are strong sub and supersolutions of \eqref{Plambda} respectively, with $\xi \leq \eta$ in $\Omega$.
Then \eqref{Plambda} has an $L^p$-viscosity solution satisfying $\xi \leq u \leq \eta$ in $\Omega$.
Furthermore,

\begin{enumerate}[(i)]
\item If $\xi$ and $\eta$ are strict in the sense of definition \ref{def2.1}, then for large $R>0$ we have
$
deg(I-\mathcal{T}_\lambda, \mathcal{S},0)=1
$
where $\mathcal{S}=\mathcal{O}\cap \mathcal{B}_R$, for
$
\mathcal{O}= \{u\in C_0^1(\overline{\Omega}); \; \xi\ll u \ll \eta \;in \; \Omega\}.
$
\item If \eqref{Hstrong} holds and $\lambda\geq 0$, there exists a minimal and a maximal solution, $\underline{u}$ and $\overline{u}$,  of \eqref{Plambda} in the sense that every $($strong$)$ solution $u$ of \eqref{Plambda} in the order interval $[\xi , \eta]$ $($i.e. such that $\xi (x)\leq u(x)\leq \eta (x)$ for all $ x\in\Omega)$ satisfies
$\xi\leq\underline{u} \leq u \leq \overline{u} \leq \eta$ in $\Omega$.
\end{enumerate}

Moreover, the conclusion is true if we replace $\calc(x)$ by $\calc(x,u)$ defined by $(\calc(x,u))_{ij} u_j =c_{ij}(x)R_a (u_j)$ for $i, j=1, \dots, n$, where $R_a$ is defined as $R_a(u_j)=u_j$ for $u\geq a$, $R_a (u_j)=a$ for $u_j<a$.
\end{teo}

\begin{proof}
Analogously to \cite[Claim 4.1]{jfa19}, we see that $\mathcal{T}_\lambda$ is completely continuous in compact intervals of $\lambda$, by using $C^{1,\alpha}$ regularity estimates in each equation.

Fix some $\lambda\in[ \Lambda_1, \Lambda_2]$ and consider $R\geq \max \{C,\|\xi\|_{E},\|\eta\|_{E} \}+1$, where $C$ is such that $\|u_i\|_{C^{1,\alpha}(\overline{\Omega})}\leq C$, $i=1, \cdots , n$, for every solution $u$ of \eqref{Plambda} which is in the order interval $[\xi,\eta]$, and for all $\lambda\in [\Lambda_1,\Lambda_2]$.
The existence of a solution in $[\xi,\eta]$ follows by constructing a modified problem $(\widetilde{P}_\lambda)$, which corresponds to the truncation made in \cite[p.1820]{jfa19} componentwise.
Then:

(a) solutions of $(\tilde{P}_\lambda)$ are fixed points of a truncated operator $\widetilde{\mathcal{T}}_\lambda$;

(b) the problems $(\tilde{P}_\lambda)$ and $({P}_\lambda)$ coincide in the order interval $[\xi,\eta]$;

(c) $\|\widetilde{\mathcal{T}}_\lambda u\|_{E} <R_0$, for all $ u \in E$, for some $R_0>R$, and $\mathrm{deg} (I-\widetilde{\mathcal{T}}_{\lambda} ,\mathcal{B}_{R_0}\,, 0)=1$.

\smallskip

Indeed, (b) follows by applying the maximum principle for each $i$.
Moreover, if $\xi , \eta$ are strict, then the degree computation in $\mathcal{S}$ is exactly the same as in \cite[p.1823]{jfa19}.

For the existence of extremal solutions under \eqref{Hstrong} we just need to note that, if $u,v$ are solutions of \eqref{Plambda},
then $\widetilde{\eta}:=\min  \{u,v\}$ is an $L^p$-viscosity supersolution of \eqref{Plambda}.
Indeed, if $\lambda\geq 0$, then $u_i$ and $v_i$ satisfy the equation $-F_i\,[w]\geq \lambda c_{i1}\widetilde{\eta_1}+\cdots +\lambda c_{in} \widetilde{\eta}_n +\langle M_i(x)Du_i(x),Du_i(x)\rangle +h_i$ in the $L^p$-viscosity sense, and so does $\widetilde{\eta}_i=\min \{u_i,v_i\}$. Once we know this, the proof of Theorem \ref{th2.1}(ii) follows as in \cite[Claim 4.5]{jfa19}.
\end{proof}

Now we work with an auxiliary problem \eqref{Plambda,k} which has no solutions for large $k$, and such that $(P_{\lambda,0})$ reduces to \eqref{Plambda}.
Fix $\Lambda_2>0$.
Recall that constants are understood as vector constants when we are dealing with the system, as in Definition \ref{def1.2}.
Then, Proposition \ref{estimatesbelow} gives us an a priori lower uniform bound $C_0$, depending on $\Lambda_2$, such that
\begin{center}
$u\geq -C_0$ for every $L^p$-viscosity  supersolution of \eqref{Plambda}, for all $\lambda\in [0,\Lambda_2]$.
\end{center}

Consider, thus, the  system
\begin{align} \tag{$P_{\lambda ,k}$} \label{Plambda,k}
\left\{
\begin{array}{rclcc}
-F[u]&=&\lambda \mathcal{C}(x)u+h(x)+\langle M(x)D u,D u\rangle +k\,\widetilde{h}(x) &\mbox{in} & \Omega \\
u &=& 0 &\mbox{on} & \partial\Omega
\end{array}
\right.
\end{align}
for $k\geq 0$, $\lambda\in [0,\Lambda_2]$. Also, if
$h^-=(h_1^-,\cdots , h^-_n)$, then $\widetilde{h}=(\widetilde{h}_1, \cdots, \widetilde{h}_n)$ is such that
\begin{align} \label{ctilde e A for k=1}
\widetilde{h}(x)=\widetilde{h}_{\Lambda_2} (x):=  h^-(x)+(A+\Lambda_2\, C_0)\, \widetilde{c}(x);\quad   \widetilde{c}=\max_{1\leq i\leq n} \sum_{j=1}^{n} c_{ij} \in L^\infty_+ (\Omega),
\end{align}
with $A:={\lambda_1}/{m}\,$, $m={\mu_1}/{\Lambda_P}\,$. Here, $\lambda_1=\lambda_1^+\left( \mathcal{L}^-(\widetilde{c}),\Omega \right)>0$ is the first  eigenvalue with weight $\widetilde{c}$ associated to the positive scalar eigenfunction $\varphi_1 \in W^{2,p}(\Omega)$ given by Proposition~\ref{exist eig for F-c geq 0}, namely
\begin{align} \label{eq exist eigen L-c}
(\mathcal{L}^-+\lambda_1 \widetilde{c}\,)\, [\varphi_1]=0 \;\textrm{ and }\;\varphi_1>0\; \mbox{ in }  \Omega, \quad
\varphi_1 =0 \mbox{ on } \partial\Omega.
\end{align}

\smallskip

Note that every $L^p$-viscosity solution of $(P_{\lambda,k})$ is also supersolution of \eqref{Plambda}, since $k \widetilde{h}\geq 0$, and so satisfies $u\geq -C_0$.
From this and \eqref{ctilde e A for k=1} we have, for all $ k\geq 1$,
\begin{align} \label{ctilde e A for all k natural}
\lambda \mathcal{C}(x)u+ h(x)+k\,\widetilde{h}(x)\geq -\Lambda_2C_0\,\widetilde{c}(x)-h^-(x)+\widetilde{h}(x)= A \widetilde{c}(x) \gneqq 0 \;\textrm{ a.e. in } \Omega .
\end{align}

\begin{lem} \label{lema Plambda,k has no solutions}
For each fixed $\Lambda_2>0$, $(P_{\lambda ,k})$ has no solutions for all $k\geq 1$ and $\lambda\in [0,\Lambda_2]$.
\end{lem}

\begin{proof}
First observe that, from \eqref{ctilde e A for all k natural}, every $L^p$-viscosity solution of $(P_{\lambda,k})$ is positive in $\Omega$ for $\lambda\in [0,\Lambda_2]$.
Let us assume by contradiction that \eqref{Plambda,k} has a solution $u$. Then it is also a solution of
\begin{align*}
\mathcal{L}^-[u] \leq  -\mu_1 |D u|^2 -A\widetilde{c}(x) \;\;\textrm{ and }\;\; u>0\;\;\mbox{ in } \Omega,
\end{align*}
and from Lemma \ref{lemma2.3arma},
$-\mathcal{L}^-[v] \geq \lambda_1 \widetilde{c}(x) v +A\widetilde{c}(x)$ and $ v>0$ in $\Omega$, using $mA=\lambda_1\,$, where $mv_i= e^{mu_i} -1 $, for $m$ and $A$ from \eqref{ctilde e A for k=1}, $i=1,\cdots , n$.
Now, since each $v_i$ is a supersolution of $-\mathcal{L}^-[v_i] \geq \lambda_1 \widetilde{c}(x) v_i +A\widetilde{c}(x)$,
thus $\underline{v}:=\min_{1\leq i \leq n} v_i$ satisfies
\begin{align}\label{eq v L-(c) to absurd}
(\mathcal{L}^-+\lambda_1 \widetilde{c}\,) [\,\underline{v}\,]\lneqq 0  \;\;\textrm{ and }\;\; v>0\;\;\mbox{ in } \Omega .
\end{align}
Then \eqref{eq exist eigen L-c}, \eqref{eq v L-(c) to absurd}, and Proposition \ref{th4.1 QB} yield $\underline{v} =t\varphi_1$ for some $t>0$. But this contradicts the first line in \eqref{eq v L-(c) to absurd}, since $(\mathcal{L}^-+\lambda_1\widetilde{c}) [\,t\varphi_1]=t(\mathcal{L}^-+\lambda_1\widetilde{c}) [\varphi_1]=0$\, in $\Omega$.
\end{proof}

When we are assuming hypothesis \eqref{Hstrong} we just say solutions to mean strong solutions of $({P}_\lambda)$. However, it is worth mentioning that sub and supersolutions, in general, are not strong, since we are considering the problem in the $L^p$-viscosity sense. In order to avoid possible confusion, we make explicit the notion of sub/supersolution we are referring to.

The next result is  important in degree arguments, bearing in mind the set $\mathcal{S}$ in Theorem~\ref{th2.1}(i). This will play the role of the strong subsolution $\xi$ in that theorem.

\begin{lem}\label{lemma 4.2}
Suppose \eqref{SC},  and \eqref{Hstrong}. Then, for every $\lambda\geq0$, there exists a strong strict subsolution  $\xi_\lambda$ of \eqref{Plambda} which is strong minimal, in the sense that every strong supersolution $\eta$ of \eqref{Plambda} satisfies $\xi_\lambda\leq\eta$ in $\Omega$.
\end{lem}

\begin{proof}
Let $K>0$ from Proposition \ref{estimatesbelow} be such that every $L^p$-viscosity supersolution $\eta$ of
\begin{align}\tag{$Q_\lambda$}
-F[\eta] \geq \lambda \mathcal{C}(x)\eta+\langle M(x)D \eta,D \eta\rangle -h^-(x)-1  \mbox{ in }\Omega , \;\;\;
\eta \geq  0 \mbox{ on } \partial\Omega
\end{align}
satisfies $\eta\geq - K$ in $\Omega$.
Let $\xi_0$ be the strong solution of the problem
\begin{align}\label{eq alpha0}
\mathcal{L}^-[\xi_0]= \lambda K\mathcal{C}(x)+h^-(x)+1  \mbox{ in }& \Omega , \;\;\;
\xi_0=0 \mbox{ on } \partial\Omega,
\end{align}
given, for example, by \cite{BuscaSirakov}.
Then, as the right hand side of \eqref{eq alpha0} is positive, by ABP, SMP and Hopf, we have $\xi_0\ll 0$ in $\Omega$.
As in \cite[Claim 6.3]{jfa19}, we see that
\begin{align}\label{claim 1 lemma 4.2}
\textrm{every $L^p$-viscosity supersolution $\eta$ of \eqref{Plambda} satisfies $\eta\geq \xi_0$ in $\Omega$.}
\end{align}
Indeed, notice that $\eta$ is an $L^p$-viscosity supersolution of $(Q_\lambda)$ and so satisfies $\eta\geq -K$.
Second, by \eqref{SC} and $M\geq 0$, $\eta$ is also an $L^p$-viscosity supersolution of
$$ -\mathcal{L}^-[\eta]\geq \lambda \mathcal{C}(x)\eta+h(x)\geq -\lambda K \mathcal{C}(x)-h^-(x) -1\quad\mathrm{in}\;\;\Omega. $$
Then $v:=\eta-\xi_0$ is an $L^p$-viscosity solution of $ \mathcal{L}^-[v] \leq 0 $, since $\xi_0$ is strong.
Further, $v\geq 0$ on $\partial\Omega$, then $v\geq 0$ in $\Omega$ by ABP, which proves \eqref{claim 1 lemma 4.2}.
Moreover, setting
$$
(\overline{\mathcal{C}}\,(x,t))_{ij}=c_{ij}(x) \;\; \mathrm{if}\; \;t_j\geq -K; \quad (\overline{\mathcal{C}}\,(x,t))_{ij}=-{K}\,c_{ij}(x)/t_j \;\; \mathrm{if} \;\; t_j<-K,
$$
we have
$0\leq (\overline{\mathcal{C}}\,(x,t))_{ij} \leq c_{ij}(x)\,$ a.e. in $\Omega$ and
$(\overline{\mathcal{C}}\,(x,t))_{ij}t_j\geq -Kc_{ij}(x)$ for all $t_j\in\real$. Then,
$$
-F[\xi_0]\leq -\mathcal{L}^-[\xi_0]\leq \lambda \,\overline{\mathcal{C}}\,(x,\xi_0)\xi_0+\langle M(x)D\xi_0,D\xi_0\rangle -h^-(x)-1,
$$
and so $\xi_0$ is a strong subsolution of $(\overline{Q}_\lambda)$, where $(\overline{Q}_\lambda)$ is the problem \eqref{Plambda} with $\mathcal{C},h$ replaced by $\overline{\mathcal{C}}=\overline{\mathcal{C}}(x,u)$, $\overline{h}=-h^--1$.
In addition, $(\overline{\mathcal{C}}(x,u))_{ij} u_j =c_{ij}(x)R_{-K} (u_j)$ for $i, j=1, \dots, n$, with $R_{-K}$ as in Theorem \ref{th2.1}.

Let $\eta_0$ be some fixed strong supersolution of \eqref{Plambda} (if it does not exist, the proof is finished). Then, by \eqref{claim 1 lemma 4.2}, we have $\xi_0\leq\eta_0$ in $\Omega$. Also, in that proof we observed that $\eta_0\geq - K$, so $\overline{\mathcal{C}}\,(x,\eta_0)\equiv \mathcal{C}(x)$ a.e. $x\in\Omega$, which implies that $\eta_0$ is a strong supersolution of $(\overline{Q}_\lambda)$. By Theorem \ref{th2.1}(iii), we obtain an $L^p$-viscosity solution $w$ of this problem, with $\xi_0\leq w \leq \eta_0$ in $\Omega$, which is strong and can be chosen as the minimal solution in the order interval $[\xi_0,\eta_0]$, by \eqref{Hstrong} and $\lambda\geq 0$.
As in \cite[Claim 6.5]{jfa19}, since $\lambda\geq 0$, we easily see that $\eta$ is a strict supersolution of \eqref{Plambda}, with  $\eta\geq w$ in $\Omega$ -- we only need to pay attention in performing the same argument in the end of the proof of Theorem 6.3 in order to have the minimum of supersolutions as a supersolution.
\end{proof}

Now we turn to the proof of theorems \ref{th1.1,1.2,1.3}, \ref{th1.5} and \ref{th1.4}.

\subsection{Proof of Theorem \ref{th1.1,1.2,1.3}}

We start with the coercive case.
Of course $\xi=u_0-\|u_0\|_{\infty}$ and $\eta=u_0+\|u_0\|_{\infty}$ are strong sub and supersolutions of the problem \eqref{Plambda}, for each $\lambda<0$, with $\xi\leq u_0 \leq \eta$ in $\Omega$.
Indeed, it is just a question of using \eqref{SC} to obtain $F[\xi]\geq  F[u_0]\geq F[\eta]$, together with $\lambda c(x)\xi \geq 0\geq \lambda c(x)\eta$.
Then Theorem \ref{th2.1} provides a solution $u_\lambda\in [\xi,\eta]$, for all $\lambda<0$.

To show $\|u_\lambda -u_0\|_E \rightarrow 0$ as $\lambda\rightarrow 0^+$,
we take an arbitrary sequence $\lambda_k \rightarrow 0^+$, and obtain -- via stability, $C^{1,\alpha}$ regularity and compact inclusion -- the existence of a limit function $u$ such that $u_k\rightarrow u$ in $E$, which is an $L^p$-viscosity solution of $(P_0)$. From the uniqueness of the solution at $\lambda=0$, $u=u_0$.

For the existence of a continuum from $u_0$, we fix $\varepsilon>0$ and look at the pair $\xi =u_0-\varepsilon$ and $\eta=u_0+\varepsilon$, which are
strong sub and supersolutions for $(P_0)$.
Since $u_0$ is the unique $L^p$-viscosity solution of the problem $(P_0)$, $\xi$ and $\eta$ are strict. Then, Theorem~\ref{th2.1}(i) and the uniqueness of the solution $u_0$ give us $\mathrm{ind}(I-\mathcal{T}_0\,,u_0)=1$.
Thus, by the well known degree theory results (see \cite[Theorem~3.3]{BR} for instance) there exists a continuum, whose components are unbounded in both directions $\real^+ \times E$ and $\real^-\times E$. This proves item~\textit{1} of Theorem \ref{th1.1,1.2,1.3}.
Item \textit{2}, in turn, is just a consequence of the a priori bounds obtained for every interval $ [\Lambda_1,\Lambda_2]$ not including the origin, and a priori estimates from below for every interval $[0,\Lambda_2]$.
\smallskip

For the multiplicity results in item \textit{3}, we notice that

(a) There exists a $\lambda_0>0$ such that $\,\mathrm{deg}(I-\mathcal{T}_\lambda\,,\mathcal{S},\,0)=1\,$, for all $ \lambda\in (0,\lambda_0)$;

(b) \eqref{Plambda} has two solutions when $\lambda\in (0,\lambda_0/2]$; \\
are both easy consequences of the topological methods used in \cite[Claim 6.7, Claim 6.9]{jfa19}, once we have a priori bounds and $C^{1,\alpha}$ estimates. Also, we exploit Lemma \ref{lema Plambda,k has no solutions} in place of \cite[Lemma 6.1]{jfa19}. This permits us to define the quantity
\begin{align*}
\bar{\lambda}:=\sup \{ \,\mu\, ; \;\forall\, \lambda\in (0,\mu), \; (P_\lambda )\;\, \mathrm{has\;at\;least\;two\;solutions}\} \in [\lambda_0/2,+\infty]
\end{align*}
and then infer that the two solutions obtained, for $\lambda\in (0,\bar{\lambda})$, satisfy the properties stated in Theorem \ref{th1.1,1.2,1.3}.

To finish the proof, we must show the statements in items \textit{3} and \textit{4} concerning ordering and uniqueness. Notice that \eqref{Hstrong} automatically implies that $u_{\lambda,1}$ and $u_{\lambda,2}$ are strong, as well as every $L^p$-viscosity solution of \eqref{Plambda}.
The uniqueness result in item~\textit{3} follows as in \cite[p.1839]{jfa19} under a convexity assumption on $F$, by exploiting Lemma~\ref{lemma 4.2} above.
The ordering is proved in the next claim.

Recall that the matrix $\mathcal{C}(x)$ is in the form \eqref{blocktriangularform}.
\begin{claim} \label{step 4 proof th1.3}
$u_{\lambda,1} \ll u_{\lambda,2}$ in at least one block, for all $\lambda\in (0,\bar{\lambda})$.
\end{claim}

\begin{proof}
Fix $\lambda\in (0,\bar{\lambda})$ and consider the strict strong subsolution $\xi=\xi_\lambda$ given by Lemma~\ref{lemma 4.2}.
Since in particular $\xi\leq u$ for every (strong) solution of \eqref{Plambda}, we can choose $u_{\lambda,1}$ as the minimal strong solution such that $u_{\lambda,1}\geq\xi$ in $\Omega$. We first note that this choice yields
\begin{align} \label{ulambda,1 leq ulambda,2 th1.3}
(u_{\lambda,1})_i\leq  (u_{\lambda,2})_i\;\;\;\mathrm{in}\;\;\Omega \quad \textrm{ for all $i=1,\cdots , n$}.
\end{align}
Otherwise there exists $x_0\in\Omega$ and one index $i$ such that $(u_{\lambda,1})_i(x_0)>(u_{\lambda,2})_i(x_0)$. Consider $u_\lambda:=\min\{ u_{\lambda,1}, u_{\lambda,2} \}\geq \xi$ in $\Omega$. Then Theorem \ref{th2.1} gives us a solution $u$ of \eqref{Plambda} such that $\xi\leq u \leq u_\lambda \lneqq u_{\lambda,1}$, which contradicts the minimality of $u_{\lambda,1}$, and implies \eqref{ulambda,1 leq ulambda,2 th1.3}.

Next define $v=u_{\lambda,2}-u_{\lambda,1}$ in $\Omega$, which is a nonnegative vector by \eqref{ulambda,1 leq ulambda,2 th1.3}. Then, since $u_{\lambda,1}$ and $u_{\lambda,2}$ are strong, $v$ satisfies, almost everywhere in $\Omega$,
\begin{align}\label{order1}
-\mathcal{L}^-[v]\geq -F[u_{\lambda,2}]+F[u_{\lambda,1}] \geq \lambda \mathcal{C}(x) v -2\mu_2 |Du_{\lambda,1}|\, |Dv|.
\end{align}
Hence, $v$ is a nonnegative strong solution of
\begin{align}\label{order2}
\textrm{$\mathcal{M}^-(D^2v)-\widetilde{b}\,|Dv|\leq 0\,$ in $\Omega$, \;\,for \, $\widetilde{b}=b+2\mu_2 \|Du_{\lambda,1}\|_\infty$.}
\end{align}

Of course $u_{\lambda,1}\neq u_{\lambda,2}$, then there exists one index $j$ such that $v_j \gneqq 0$ in $\Omega$. Consider the block from where it belongs; say the first one, $j\in \{ 1,\dots, t_1 \} $.
So, by \eqref{order2} and SMP, $v_j>0$ in $\Omega$.
Now look at the $j$-th column of this block. By \eqref{blocktriangularform} we know that there exists an index $k \ne j$, $k \in \{ 1,\dots, t_1 \}$, such that $c_{kj} \ne 0$.

Finally, let us turn back to \eqref{order1}, and consider the $k$-th equation of it. Since $c_{kj}v_j \gneqq 0$, by \eqref{order2} and SMP we obtain that $v_k>0$ in $\Omega$.
Using the full coupling of $\mathcal{C}(x)$, we can iterate this process $t_1$ times, by visiting all the equations. Therefore $v_j > 0$ for all $j \in \{ 1,\dots, t_1 \}$. Applying Hopf, we conclude that $v\gg 0$ in this block.
\end{proof}


\subsection{Proof of Theorems \ref{th1.5} and \ref{th1.4}}

Both results are an easy extension of considerations made in \cite{jfa19}, as long as we exploit Lemma \ref{lemma 4.2} instead of \cite[Lemma 6.2]{jfa19}.
In particular, for Theorem \ref{th1.5} we just need to be careful when applying the SMP, as we make explicit in the next lemma -- which is the extension to a system  of \cite[Lemma 6.14]{jfa19}.

\begin{claim} \label{step 1 th1.5}
$u_0$ is a strict strong supersolution of \eqref{Plambda}, for all $\lambda>0$.
\end{claim}

\begin{proof}
Since $\lambda \mathcal{C}(x)u_0\lneqq 0$ in $\Omega$, $u_0$ is a strong supersolution of \eqref{Plambda}. To see that it is strict, we take $u\in E$ an $L^p$-viscosity subsolution of \eqref{Plambda} such that $u\leq u_0$ in $\Omega$, and set $U:=u_0-u$. Then, since $u_0$ is strong, $U$ is an $L^p$-viscosity supersolution of
\begin{align*}
-\mathcal{L}^-[U]&\geq \lambda \mathcal{C}(x) U -\langle M(x)D U,D U\rangle +\langle M(x)D u_0,D U\rangle+\langle M(x)D u_0,D U\rangle\\
&\geq - \mu_2\, |DU|^2-2\mu_2\, |Du_0|\,|DU|,
\end{align*}
so $\mathcal{\widehat{L}}^-[w]\leq 0$ in $\Omega$ in the $L^p$-viscosity sense, where
\begin{align}\label{def L-hat}
\mathcal{\widehat{L}}^- [w]:=\mathcal{M}^-(D^2w)-\widehat{b}\,|Dw|, \;\; \textrm{ for \; $\widehat{b}=b+2\mu_2 \,\|Du_0\|_\infty$,  }
\end{align}
and $mw_i=1-e^{-mU_i}$, $m=\mu_2/\lambda_P$, by Lemma \ref{lemma2.3arma}, for $i=1,\cdots , n$.

Assume that there exists an index $j$ in the first block $\{ 1,\dots, t_1 \}$ such that $w_j(x_0)=0$.
Then by SMP we have $w_j \equiv 0$, hence $U_j \equiv 0$. Let us turn back to \eqref{Plambda}, and consider the $j$-th equation. By \eqref{blocktriangularform} we know that there exists an index $k \ne j$, $k \in \{ 1,\dots, t_1 \}$, such that $c_{jk} \ne 0$. This, combined with $U_j \equiv 0$, implies $U_k(x_1)=0$ for some point $x_1$. We now apply again SMP, to get $U_k \equiv 0$. As each diagonal block in $\mathcal{C}(x)$ is fully coupled, we can iterate $t_1$ times, and visit all the equations, therefore $U_j \equiv 0$ for any $j \in \{ 1,\dots, t_1 \}$.
However, hypothesis \eqref{cu0 nonzero} provides a contradiction, and hence $U_j > 0$ for all $j \in \{ 1,\dots, t_1 \}$.
Taking into account each block separately, and applying Hopf, we conclude $U\gg 0$.
\end{proof}

As for Theorem  \ref{th1.4}, showing that every nonnegative supersolution in $E$ of \eqref{Plambda} for $\lambda>0$ satisfies $u \gg u_0$ follows by analogous considerations to those made in the proof of Claim \ref{step 1 th1.5} above.
Everything else works as in the scalar case, up to obvious modifications. The only point which requires some attention in our multiplicity analysis is the analog of Claim 6.20 in \cite{jfa19} which is our Claim~\ref{no nonnegative sol of Plambda for lambda large} ahead.
Recall that nonexistence type results were obtained in Lemma \ref{lema Plambda,k has no solutions} via \eqref{ctilde e A for all k natural}. There, the possibility of taking a large parameter $k$ overcame the difficulty.
Here we have a different situation because we need to conclude the existence of two distinct positive solutions without using Proposition \ref{lemma 4.2} -- note that in Theorem \ref{th1.5} it is simpler as soon as we have $u_0$ as supersolution.
Therefore we need to work with problem \eqref{Plambda} itself, in which nonexistence for the system does not seem to be a consequence of the scalar framework, at least not in the general case.

\begin{claim} \label{no nonnegative sol of Plambda for lambda large}
\eqref{Plambda} has no nonnegative $L^p$-viscosity supersolutions for $\lambda$ large.
\end{claim}

\begin{proof}
Consider the matrix $\mathcal{C}(x)$ in the form \eqref{blocktriangularform}.

\smallskip

Let $\lambda\geq \widehat{\lambda}_1$, where $\widehat{\lambda}_1={\lambda}_1^+\, ( \widehat{\mathcal{L}}^-(\widehat{c}\,),\Omega )>0$ is the principal eigenvalue of the operator $\mathcal{\widehat{L}}^-$ defined in \eqref{def L-hat}, but now with weight $\widehat{c}(x)\gneqq 0$, where
$$
 \widehat{c} (x)=\min_{1\leq i\leq t_1 } \sum_{j=1}^n c_{ij}(x)  \;\; \textrm{a.e. in } \Omega, \;\;\textrm{ with $t_1$ from \eqref{blocktriangularform}},
$$
which is associated to the positive eigenfunction $\widehat{\varphi}_1={\varphi}_1^+\,( \widehat{\mathcal{L}}^-(\widehat{c}\,),\Omega )\in W^{2,p}(\Omega)$, that is,
\begin{align} \label{def phi1tilde+ of L-(c) ch6}
(\widehat{\mathcal{L}}^-+\widehat{\lambda}_1\,\widehat{c}\,) \,[\,\widehat{\varphi}_1] = 0 \textrm{ and }  \;\;\widehat{\varphi}_1 >0  \mbox{ in } \Omega
, \;\;\widehat{\varphi}_1=0  \mbox{ on } \partial\Omega.
\end{align}

Notice that if $t_1=1$, then $ \widehat{c} (x)=c_{11}(x)$ which is nontrivial by hypothesis \eqref{1x1 blocks nonzero}.

Suppose, then, in order to obtain a contradiction, that there exists a nonnegative $L^p$-viscosity supersolution $u$ of \eqref{Plambda} and set $v=u-u_0$ in $\Omega$.
One proves $v\gg 0$ in $\Omega$ by performing the same SMP argument done in Claim \ref{step 1 th1.5}.
Now, since $u_0$ is strong, we can use it as a test function into the definition of $L^p$-viscosity supersolution of $u$, to obtain
\begin{align*}
-\mathcal{L}^-[v]&\geq \lambda \mathcal{C}(x)v+\lambda \mathcal{C}(x) u_0 +\langle M(x)D v, D v\rangle +\langle M(x)D v, D u_0\rangle +\langle M(x)D u_0, D v\rangle \\
&\gneqq \widehat{\lambda}_1 \mathcal{C}(x) v-2\mu_2 |D u_0|\, |D v|,
\end{align*}
using $\mathcal{C}(x)u_0\gneqq0$. Then each $v_i$ satisfies
$-\widehat{\mathcal{L}}^-[v_i]\gneqq \widehat{\lambda}_1\, \widehat{c}(x) \,\underline{v}$\; in $\Omega$, for $i=1,\cdots, t_1$, in the $L^p$-viscosity sense, where
$\underline{v}:=\min_{1 \leq i \leq t_1 } v_i$, since $\lambda , c_{ij}, v_i\geq 0$. Hence,
\begin{align}\label{eq v L-(c) to absurd th1.4}
(\widehat{\mathcal{L}}^-+\widehat{\lambda}_1 \,\widehat{c}\,)\, [\,\underline{v}\,] \lneqq  0 \;\;\textrm{ and }\;\; \underline{v}> 0 \;\mbox{ in }  \Omega
\end{align}
Thus we apply Proposition \ref{th4.1 QB} to \eqref{def phi1tilde+ of L-(c) ch6} and \eqref{eq v L-(c) to absurd th1.4}, from where $\underline{v}=t\widehat{\varphi}_1$ for some $t>0$. But this contradicts  \eqref{eq v L-(c) to absurd th1.4}, since $(\widehat{\mathcal{L}}^-+\widehat{\lambda}_1\,\widehat{c}\,) \,[\,t\widehat{\varphi}_1]=0$ in $\Omega$.
\end{proof}

In the next section we prove the second part of Theorem \ref{th1.4} only in the scalar case $n=1$, since the extension to systems can be established as above.

\section{Complementary multiplicity for scalar equations}\label{section scalar}

Here and in the next section, $E=C^1(\overline{\Omega})$. Now we consider the scalar problem
\begin{align}\label{Plambda,gamma} \tag{$P_{\lambda,\gamma}$}
\left\{
\begin{array}{rclcc}
-F[u] &=&\lambda c(x)u+\langle M(x)D u, D u \rangle +\gamma  h(x) &\mbox{in} & \Omega  \\
u&=& 0 &\mbox{on} & \partial\Omega
\end{array}
\right.
\end{align}
where $\Omega$ is a bounded $C^{1,1}$ domain in $\rN$, $\lambda\in \mathbb{R}$, $\gamma> 0$, $N\geq 1$, $c,h\in L^p(\Omega)$, $c\gneqq 0$, $M$ is a bounded matrix, and $F$ is a fully nonlinear uniformly elliptic operator which satisfies \eqref{SC}, \eqref{ExistUnic M bem definido}, and \eqref{Hstrong}.
The results in this section  are  related to \cite{siam2010} and in particular extend to nondivergence form equations \cite[Corollary 1.9]{CJ}, where variational problems were considered.

\smallskip

By Theorem 1(ii) of \cite{arma2010}, there exists $\Gamma_0>0$ such that the problem $(P_{0,\gamma})$ has an $L^p$-viscosity solution, namely $u_{0,\gamma}$, for each $\gamma\in [0,\Gamma_0]$. Note that $u_{0,\gamma}$ is strong by regularity, and so unique by Theorem 1(iii) of \cite{arma2010}.

\smallskip

Say that $h\gneqq 0$, then $u_{0,\gamma}\geq 0$, with $c(x)u_{0,\gamma}\gneqq 0$, for all $\gamma\in (0,\Gamma_0]$ (see Remark 6.25 in \cite{jfa19}).
Thus, there exists $\overline{\lambda}_1 >0$ such that \eqref{Plambda,gamma} has at least two positive solutions for $\lambda\in (0,\bar{\lambda}_1)$, it has at least one nonnegative strong solution at $\lambda=\bar{\lambda}_1$, and no nonnegative $L^p$-viscosity solutions for $\lambda> \bar{\lambda}_1$.

\medskip

Let $\lambda\geq {\lambda}_1^-$, where
${\lambda}_1^-:={\lambda}_1^-\, ( {\mathcal{L}}^+(c),\Omega )>0$
is the principal positive weighted eigenvalue of $\mathcal{L}^+$ associated to the negative eigenfunction
${\varphi}_1^-:={\varphi}_1^-\,( {\mathcal{L}}^+(c),\Omega )\in W^{2,p}(\Omega)$
from Proposition \ref{exist eig for F-c geq 0}, that is,
\begin{align} \label{def phi1tilde+ of L-(c)}
({\mathcal{L}}^+ +{\lambda}_1^- c) [\varphi_1^-] = 0  \;\textrm{ and }\;{\varphi}_1^- <0\;  \mbox{ in } \Omega , \quad
\varphi_1^- =0  \mbox{ on } \partial\Omega.
\end{align}
Notice that, since $\mathcal{L}^+$ is convex, then
\begin{align}\label{lambda1+ leq lambda1-}
\lambda_1^+:=\lambda_1^+\, ( {\mathcal{L}}^+(c),\Omega )\leq {\lambda}_1^-\, ( {\mathcal{L}}^+(c),\Omega )=\lambda_1^-.
\end{align}

\begin{claim}\label{lambdabar1 cota}
$\bar{\lambda}_1<\lambda_1^-$.
\end{claim}
In other words, Claim \ref{lambdabar1 cota} says that  \eqref{Plambda,gamma} does not admit nonnegative solutions for $\lambda\geq \lambda^-_ 1$.
To see this, we observe that if a such solution $u$ existed, since $\gamma h\gneqq0$, then $u$ would satisfy
$
-\mathcal{L}^-[u]\gneqq {\lambda}_1^- c(x) u,
$
so $u>0$ in $\Omega$ by SMP.
But then this strict inequality combined with Proposition \ref{th4.1 QB} and \eqref{def phi1tilde+ of L-(c)} produces $u=t{\varphi}_1^-$ for some $t>0$, a contradiction.

\begin{teo}\label{th cite introd}
There exists a positive $\Gamma\leq\Gamma_0$ such that, for each $\gamma\in (0,\Gamma)$, we have the existence of $\bar{\lambda}_2>0$ for the problem \eqref{Plambda}=\eqref{Plambda,gamma} satisfying
\begin{enumerate}[(i)]
\item for $\lambda> \bar{\lambda}_2$, \eqref{Plambda} has at least two solutions with $u_{\lambda,1}\ll 0$ in $\Omega$ and $\min_{\overline{\Omega}} u_{\lambda,2}<0$;

\item for $\lambda= \bar{\lambda}_2$, \eqref{Plambda} has at least one nonpositive solution, which is unique if $F$ is convex;

\item for $\lambda < \bar{\lambda}_2$, the problem \eqref{Plambda} has no nonpositive solution.
\end{enumerate}
\end{teo}

\begin{proof}
Firstly we are going to prove that there exists $\Gamma>0$ such that the problem $(P_{\lambda_0,\gamma})$ has a nonpositive supersolution $\eta_\gamma$, for all $\gamma\in (0, \Gamma)$, where $\lambda_0 $ is some positive number independent of $\lambda$ and $\gamma$.

\smallskip

Let $w $ be some (fixed) strong solution of
\begin{align}\label{Q}
\left\{
\begin{array}{rclcc}
-\mathcal{L}^+[w] &=&\lambda_0 \, c(x)w+1+ h(x) &\mbox{in} & \Omega  \\
w&=& 0 &\mbox{on} & \partial\Omega
\end{array}
\right.
\end{align}
for some $\lambda_0\in (\lambda_1^-,\lambda_1^-+\varepsilon_0)$, $\varepsilon_0>0$.
The existence of $w$ is ensured by Theorem \ref{th Exist>lambda1}, since the operator $\mathcal{L}^+$ satisfies the $W^{2,p}$ regularity hypothesis \eqref{Hstrong}.
\smallskip

Then, let $C_0 >0$ be such that $\|Dw\|_\infty^2 \leq C_0$, and set
$\Gamma:=\min \{ \Gamma_0, (\mu_2\, C_0)^{-1}\}$.

\begin{claim}\label{AMP}
Up to taking a smaller $\varepsilon_0$, we have $w\ll 0$ in $\Omega$.
\end{claim}

Assuming Claim \ref{AMP}, we define $\eta=\eta_\gamma : = \gamma w $, for $0<\gamma\leq \Gamma $, which is a negative function.
Then we have, in the $L^p$-viscosity sense,
\begin{align*}
-F[\eta] & \geq -\mathcal{L}^+[\eta] =\lambda_0\, c(x)\eta +\gamma+\gamma h(x)
\geq \lambda_0\, c(x)\eta +\gamma^2 \mu_2 \, C_0+\gamma h(x) \\
&\geq \lambda_0\, c(x)\eta +\langle M(x)D\eta , D\eta\rangle +\gamma h(x).
\end{align*}
That is, $\eta$ is a supersolution of $(P_{\lambda_0,\gamma})$, for all $\gamma\in (0,\Gamma)$, with $\eta \ll 0$ in $\Omega$.

\begin{proof}[Proof of Claim \ref{AMP}]
We are going to prove a stronger result, i.e.\ that there exists  a small $\varepsilon_0 >0$ such that every solution $w\in E $ of \eqref{Q} satisfies $w< 0$ in $\Omega$ -- which in turn yields $w\ll 0$ in $\Omega$, by Hopf.

Assume the contrary, then there exists a sequence $\lambda_k \rightarrow \lambda_1^-$ and $w_k$ satisfying
\begin{align}
\left\{
\begin{array}{rclcc}
-\mathcal{L}^+[w_k] & = & \lambda_k\,c(x) w_k +f(x) &\mbox{in} & \Omega  \\
w_k&=& 0 &\mbox{on} & \partial\Omega ,
\end{array}
\right.
\end{align}
but each $w_k$ is such that
\begin{align}\label{contrad x0}
\textrm{$\max_{\overline{\Omega}} w_k =w_k(x_k)\geq 0$, where $x_k\in\Omega$, and $Du(x_k)=0$, for all $k$.}
\end{align}
By taking a subsequence, $x_k\rightarrow x_0\in \overline{\Omega}$. Since $f\not\equiv 0$, of course $w_k\not\equiv 0$, for all $k$.

We claim that there is a subsequence such that
\begin{align}\label{unif norm unbounded}
\|w_k\|_\infty \rightarrow \infty .
\end{align}
Indeed, if this was not the case, $\|w_k\|_\infty \leq C$, for some positive constant $C$ independent of $k$.
By $C^{1,\alpha}$ regularity, compact inclusion and stability, this would give us some $w\in E$, which is a viscosity solution of
\begin{align*}
\left\{
\begin{array}{rclcc}
-\mathcal{L}^+[w] & = & \lambda_1^-\,c(x) w +f(x) &\mbox{in} & \Omega  \\
w&=& 0 &\mbox{on} & \partial\Omega .
\end{array}
\right.
\end{align*}
Now, if $w$ was nonnegative in $\Omega$, it should be positive by SMP; then $\lambda_1^-\leq \lambda_1^+$ by the definition of $\lambda_1^+$. Hence $\lambda_1^-= \lambda_1^+$ by \eqref{lambda1+ leq lambda1-}. Proposition \ref{th4.1 QB} would imply so $w=t\varphi_1^+$, for some $t>0$, which contradicts $f\neq 0$.
Thus, we must have $w(x_1)<0$ for some $x_1\in \Omega$. This yields $w=t\varphi^-_1$, $t>0$ by Proposition \ref{th4.1 QB}, contradiction. Thus, \eqref{unif norm unbounded} holds.

\smallskip

Then, for the sequence in \eqref{unif norm unbounded}, we define $v_k:={w_k}/{\|w_k\|_\infty}$, which satisfies
\begin{align*}
\left\{
\begin{array}{rclcc}
-\mathcal{L}^+[v_k] & = & \lambda_k\,c(x) v_k +f/{\|w_k\|_\infty} &\mbox{in} & \Omega  \\
v_k&=& 0 &\mbox{on} & \partial\Omega .
\end{array}
\right.
\end{align*}
Since $\|v_k\|_{C^{1,\alpha}(\overline{\Omega})}\leq C$, then passing to a subsequence, $v_k$ converges in $E$ to some function $v$, which is a solution of
$-\mathcal{L}^+[v]  = \lambda_1^-c(x) v  $ in $\Omega$,  $v=0$ on $\partial\Omega$, by stability.
Note that $\|v\|_\infty = \lim_{k}| v_k (y_k)| =1$, for some sequence of points $y_k\in \overline{\Omega}$.

If we had $v(x_1)<0$ for some $x_1\in \Omega$, by Proposition \ref{th4.1 QB} we would obtain  $v=\varphi^-_1 <0$. Thus, by \eqref{contrad x0}, $v(x_0)=0$ and $x_0\in \partial\Omega$. So the application of Hopf at $x_0$ contradicts \eqref{contrad x0}.

Therefore, we must have $v\geq 0$ in $\Omega$, i.e.  $v>0$ in $\Omega$ by SMP. Then  $\lambda_1^-=\lambda_1^+$, by the definition of $\lambda_1^+$ and \eqref{lambda1+ leq lambda1-}. Hence, Proposition \ref{th4.1 QB} yields $v=\varphi_1^+>0$ in $\Omega$.
Now Hopf gives us $\partial_\nu v >0$ on $\partial\Omega$. This fact and the convergence of $v_k$ to $v$ in $E$ imply that $v_k>0$ in $\Omega$ for large $k$.
Therefore, for large $k$, $v_k$ is a solution of
\begin{align*}
-\mathcal{L}^+[v_k]  \gneqq  \lambda_1^+c(x) v_k  \;\textrm{ and } \;v_k> 0 \mbox{ in }  \Omega , \quad v_k= 0 \mbox{ on }  \partial\Omega .
\end{align*}
Thus $v_k=t\varphi_1^+$, for some $t>0$, by Proposition \ref{th4.1 QB} again. The above strict inequality finally provides the last contradiction, and proves Claim \ref{AMP}.
\end{proof}

Next let us fix some $\gamma\in (0,\Gamma]$ and look at the problem \eqref{Plambda} = \eqref{Plambda,gamma}.

Recall that \eqref{Plambda} has a strong strict subsolution $\xi_\lambda$ for all $\lambda\geq 0$.
However, notice that our $\eta$ constructed above, besides being a supersolution for only a fixed $\lambda_0$,  has no reason to be strict.
Nevertheless, we can check that a slight variation of the argument in the proof of Theorem 1.7 in \cite{CJ} ensures the strictness for an arbitrary $\lambda$ and enables us to use Theorem \ref{th2.1}.
For the sake of completeness, we give the details at the points in which the general context of $L^p$-viscosity solutions requests an extra care.

Note that $c(x)\eta \lneqq 0$ in $\Omega$.
Otherwise the problem $(P_{0})$ would have a solution $v$ such that $\xi_0\leq v\leq \eta< 0$, due to Lemma \ref{lemma 4.2} and the first part of Theorem \ref{th2.1}.
Then we define
\begin{align*}
\bar{\lambda}_2:=\inf \{\lambda \geq 0; \; \eqref{Plambda} \textrm{ has a strong supersolution } \eta_{\lambda}\leq 0 \textrm{ with } c(x)\eta_\lambda \lneqq 0 \;\} \leq \lambda_0.
\end{align*}

\smallskip

Let $\lambda>\bar{\lambda}_2$, then there exists $\tilde{\lambda}\in (\bar{\lambda}_2,\lambda)$ such that \eqref{Plambda} has a strong supersolution
$\eta_{\tilde{\lambda}}\leq 0$ with $c(x)\eta_{\tilde{\lambda}} \lneqq 0$.
But now $\eta_{\tilde{\lambda}}$ is a strong supersolution of $(P_\lambda)$, which is not a solution. So, proceeding as in Theorem 2.3 in \cite{jfa19} we see that $\eta$ is strict.
Then we use Theorem \ref{th2.1}(i) to obtain that $\mathrm{deg} (I-\mathcal{T}_\lambda, \mathcal{S}_\lambda,0)=1$, where
$$
\mathcal{S}_\lambda=\{ \xi_\eta \ll u \ll \eta_{\tilde{\lambda}} \}\cap \mathcal{B}_R,
$$
for some $R>0$. This gives us the first solution $u_{\lambda,1}\ll 0$. Thus, for $\lambda$ small, a second solution $u_{\lambda,2}$ satisfying $u_{\lambda,2}\gg u_{\lambda,1}$ is also established as in the scalar case, as well as the monotonicity of $u_{\lambda,1}$ with respect to $\lambda$, see \cite[Claim 6.9, Claim 6.12]{jfa19}.

On the other hand, if $\lambda > \bar{\lambda}_2$, we can only have a nonpositive solution $u$ satisfying $c(x)u\equiv 0$. In such a case, $\gamma h \gneqq 0$ and an exponential change from Lemma \ref{lemma2.3arma} generates a nonpositive solution of
$
\mathcal{L}^+[v] \lneqq 0
$
in $\Omega$, and $v<0$ in $\Omega$ by SMP. Since $\lambda^-_1 c(x)v\equiv 0$, these inequalities and \eqref{def phi1tilde+ of L-(c)}, in the application of Proposition \ref{th4.1 QB}, yield a contradiction.

Observe that $\bar{\lambda}_2$ cannot be zero by Remark 6.22 in \cite{jfa19}. Indeed, via eigenvalue arguments it was shown there that, for small values of $\lambda$, every solution must be nonnegative.

To finish, we notice that a sequence $\lambda_k \rightarrow \bar{\lambda}_2$ produces a sequence $u_{\lambda_k,1}$ of negative solutions of $(P_{\lambda_k})$. Then, a priori bounds on $[\bar{\lambda}_2, \bar{\lambda}_2+1]$, $C^{1,\alpha}$ estimates, compact inclusion and stability ensure the existence of an $L^p$-viscosity solution $u$ of $(P_{\bar{\lambda}_2})$, which is nonpositive by convergence, and strong by \eqref{Hstrong}.
This completes the proof.
\end{proof}

\begin{remark}
If $F$ is convex, 1-homogeneous and possesses eigenvalues, for instance if $F=\mathcal{L}^+$ or a HJB operator, then the estimate can be improved. In fact, in this case in Claim \ref{lambdabar1 cota} we use $\lambda_1^+(F(c))$ instead of $\lambda_1^+(\mathcal{L}^-(c))$, which gives us $$\bar{\lambda}_1<\lambda_1^+(F(c))\leq \lambda_1^-(F(c))<\bar{\lambda}_2.$$
\end{remark}


\section{A short miscellaneous on weighted eigenvalues}\label{section eigenvalue}

We consider the more general structure
\begin{align}\mathcal{M}_{\lambda, \Lambda}^- (X-Y)-b(x)|\vec{p}-\vec{q}|-d(x)\,\omega ((r-s)^+) \leq F(x,r,\vec{p},X) - F(x,s,\vec{q},Y) \label{SCG}\tag{$SCG$}  \\ \leq \mathcal{M}_{\lambda, \Lambda}^+ (X-Y)+b(x)|\vec{p}-\vec{q}|+d(x)\,\omega ((s-r)^+)\;\; \textrm{ a.e. } x\in \Omega \nonumber\end{align}
with $F(\cdot,0,0,0)\equiv 0$, where $0<\lambda \leq \Lambda$,\, $ b\in L^p_+ (\Omega)$, $ p>N$, $d\in L^\infty_+(\Omega)$, $\omega$ a Lipschitz modulus.
Here, the condition over the zero order term in \eqref{SCG} means that $F$ is proper/coercive, i.e.\ nonincreasing in $r$.
On $F$ we also impose  \eqref{ExistUnic M bem definido}, and 1-homogeneity such as
\begin{align}\label{homogeneity}
\textrm{$F(x,tr,t\vec{p},tX)=tF(x,r,\vec{p},X)$ for all $t\geq 0$.}
\end{align}

Notice that solvability in $L^N$-viscosity sense was used in \cite{regularidade}, but this notion is equivalent to solvability in $L^p$-sense from \eqref{ExistUnic M bem definido}, once we have the data $f$ in $L^p (\Omega)$, see \cite{tese}.

\smallskip

For any $c\in L^p(\Omega)$, with $c\gneqq 0$ and $p> N$, and $F$ satisfying the above assumptions, we can define, as in \cite{BNV, regularidade, QB},
\begin{align*}
\lambda_1^\pm=\lambda_1^\pm\,(F(c),\Omega)&=\sup\left\{ \lambda>0; \; \Psi^\pm(F(c),\Omega,\lambda)\neq \emptyset\right\}
\end{align*}
where
$$\Psi^\pm (F(c),\Omega,\lambda):=\left\{ \psi\in C(\overline{\Omega}); \; \pm\psi>0 \textrm{ in }\Omega,\; \pm (F[\psi]+\lambda c(x)\psi )\leq 0 \textrm{ in }\Omega \right\};$$
with inequalities holding in the $L^p$-viscosity sense (equivalent to $L^N$). Notice that, by definition,
$\lambda_1^\pm (G(c),\Omega)= \lambda_1^\mp (F(c),\Omega),$
where $G(x,r,p,X):=-F(x,-r,-p,-X)$.

We recall the following result on existence of eigenvalues with nonnegative unbounded weight, from  \cite{regularidade}.

\begin{teo}\label{exist eig for F-c geq 0}
Let $\Omega\subset\rN$ be a bounded $C^{1,1}$ domain, $c\in L^p(\Omega)$, $c\gneqq 0$ for $p>n$, $F$ as above, for $b,\, d\in L^\infty_+ (\Omega)$. Then $F$ has two positive weighted eigenvalues $\alpha_1^\pm>0$ corresponding to normalized and signed eigenfunctions $\varphi_1^\pm\in C^{1,\alpha}(\overline{\Omega})$ that satisfy
\begin{align} \label{eq exist eigen F c lambda1+}
\left\{
\begin{array}{rclcc}
F[\varphi_1^\pm]+\alpha_1^\pm c(x) \varphi_1^\pm &=& 0 &\mbox{in} & \Omega \\
\pm \varphi_1^\pm &>& 0 &\mbox{in} &\Omega \\
\varphi_1^\pm &=& 0 &\mbox{on} &\partial\Omega
\end{array}
\right.
\end{align}
in the $L^p$-viscosity sense, with $\max_{\overline{\Omega}}\,(\pm \varphi_1^\pm) =1$.
If, moreover, the operator $F$ satisfies \eqref{Hstrong}, then $\alpha_1^\pm =\lambda^\pm_1$ and the conclusion is valid also for $b\in L^p_+ (\Omega)$.
\end{teo}

Of course, Pucci's extremal operators
$\mathcal{L}^\pm $, with $b\in L^p_+(\Omega)$,
are examples of $F$ which satisfy \eqref{Hstrong}. Such existence results for $\mathcal{L}^\pm$ are used several times in the text.

The following proposition for unbounded $c$  is both an auxiliary result for the proof of Theorem \ref{exist eig for F-c geq 0} and  an important tool for proving nonexistence results for equations in  nondivergence form.

\begin{prop}\label{th4.1 QB} Let $u,v\in C(\overline{\Omega})$ be $L^p$-viscosity solutions of
\begin{align}\label{eq th1.4 2case ineq}
\left\{
\begin{array}{rclcc}
F[u]+c(x)u &\geq & 0 & \mbox{in} &\Omega \\
u &<& 0& \mbox{in} & \Omega
\end{array}
\right. ,\quad
\left\{
\begin{array}{rcll}
F[v]+c(x)v &\leq & 0 &\mbox{in} \;\;\; \Omega \\
v &\geq & 0 & \mbox{on} \;\; \partial\Omega\\
v(x_0) &< & 0  &x_0 \in\Omega
\end{array}
\right.
\end{align}
with $F$ as above, $c\in L^p(\Omega)$, $p>n$. Suppose one, $u$ or $v$, is a strong solution. Then, $u=tv$ for some $t>0$.
The conclusion is the same if $F[u]+c(x)u\leq 0$, $F[v]+c(x)v\geq 0 $ in $\Omega$, with $u>0$ in $\Omega$, $v\leq 0$ on $\partial\Omega$ and $v(x_0)>0$ for some $x_0 \in\Omega$.
\end{prop}

A consequence of the proof of our Claim \ref{AMP} is an improved version of the \textit{anti-maximum principle} \cite{Arms2009}. We state it for the sake of completeness. Consider the problem
\begin{align}\label{Flambda}
\textrm{$F[u]+\lambda c(x) u =f(x)$ in $\Omega$, \;\;\;$u=0$ on $\partial\Omega$.}
\end{align}
Recall that solutions of this problem are at least $C^{1,\alpha}$ up to the boundary for $\Omega\in C^{1,1}$.

\begin{cor}
Let $f\in L^p(\Omega)$, with $p>N$ and $f\gneqq 0$. Then then there exists $\varepsilon_0 >0$ such that any solution $u$ of \eqref{Flambda}, with $\lambda\in (\lambda_1^-(F(c),\Omega),\lambda_1^-(F(c))+\varepsilon_0)$, satisfies $u<0$ in $\Omega$.
An analogous result holds if $f\lneqq 0$, related to $\lambda_1^+(F(c),\Omega)$ and positive solutions.
\end{cor}

\smallskip

We finally turn to the main result of this section, concerning existence for the Dirichlet problem. This result is needed, for instance, to ensure existence of solutions of \eqref{Q}.
We give a proof of it in the sequel, following the ideas of \cite{Arms2009, siam2010}, in the context of $L^p$-viscosity solutions, for fully nonlinear equations with unbounded coefficients.

For ease of notation, we will be omitting the information $(F(c),\Omega)$ each time in what follows.
Consider $\lambda_1:=\max\{ \lambda_1^+, \lambda_1^-\}$.
Then define, as in \cite{Arms2009}, the following quantity
\begin{align*}
\lambda_2 (F(c),\Omega):=\inf \{ \rho > \lambda_1 \textrm{ such that $\rho$ is an eigenvalue of $F$ in $\Omega$, with weight $c$}\}.
\end{align*}
Notice that $\lambda_2(F(c),\Omega)= +\infty$ is possible.

\begin{teo}\label{th Exist>lambda1}
Assume \eqref{SCG},  \eqref{ExistUnic M bem definido}, \eqref{Hstrong}, and \eqref{homogeneity}.
Let $f\in L^p (\Omega)$, with $p> N$, and let $\lambda_1 <\lambda < \lambda_2$.
Then there exists a strong solution of the Dirichlet problem \eqref{Flambda}.
\end{teo}

\begin{proof}
We define
$
F_\tau [u] = \tau F [u] + ( 1 - \tau ) \Delta u
$
for $u\in E$, which satisfies \eqref{SCG},  \eqref{ExistUnic M bem definido}, \eqref{Hstrong}, and \eqref{homogeneity}. Then, from Theorem \ref{exist eig for F-c geq 0}, we write
$\lambda^-_\tau = \lambda_1^- (F_\tau(c),\Omega)$, associated to
$\varphi_\tau =\varphi_1^-(F_\tau (c), \Omega)$, which is such that $\varphi_\tau\le 0$ and $\|\varphi_\tau \|_\infty =1$,  for all $\tau\in [0, 1]$.

We first claim that the function $\tau \mapsto \lambda^-_\tau $ is continuous in the interval $[0, 1]$. Indeed, let $\tau_k \in [0,1]$, $\tau_k\rightarrow \tau_0$.  Hence it follows that the sequence $\lambda_{\tau_k}^-$ is bounded, by the same procedure done in the proof of Theorem 5.2 in \cite{regularidade}.
So, passing to a subsequence, we can say that $\lambda_{\tau_k}^- \rightarrow \lambda_0$ for some $\lambda_0$.
Then, by $C^{1,\alpha}$ estimates, compactness argument and stability, we obtain a solution $\varphi_0\in E$ of \eqref{Flambda} with  $\lambda=\lambda_0$. Notice that $\varphi_0\le 0$ and $\|\varphi_0\|_{\infty}=1$.
By the simplicity of the eigenvalues (which is true under hypothesis \eqref{Hstrong}, see \cite{regularidade}), we have $\lambda_0=\lambda^-_{\tau_0}$, and so the continuity follows.
Analogously, $\tau \mapsto \lambda^+_\tau $ is continuous, where
$\lambda^+_\tau = \lambda_1^+ (F_\tau (c),\Omega)$.

On the other hand, we infer that the map $\tau \mapsto \bar{\lambda}_\tau $, given by  $\bar{\lambda}_\tau = \lambda_2 (F_\tau (c), \Omega)$, is lower semicontinuous; and therefore, for each $\lambda\in (\lambda_1,\lambda_2)$, we guarantee the existence a continuous function $\mu_\tau$ in $[0,1]$ satisfying
$\mu_0 = \lambda$, and
$\lambda_\tau \leq \mu_\tau \leq \bar{\lambda}_\tau$, for all $\tau\in [0,1]$,
Here, $\lambda_\tau = \max \, ( \lambda^-_\tau , \lambda^+_\tau )$.
In fact, this is accomplished by using arguments similar those in Propositions 5.5 and 5.6 of \cite{Arms2009} -- the slight differences have already appeared in the proof of Claim \ref{AMP}.

Next we define the operator $\mathcal{A}_\tau \,: E\rightarrow E$ which takes a function $u$ into $\mathcal{A}_\tau u =U$, where $U$ is the unique $L^p$-viscosity solution of the problem
\begin{align*}
F_\tau\, [U] = \mu_\tau c(x)u +f(x) \textrm{ in } \Omega, \quad U=0 \textrm{ on } \partial\Omega.
\end{align*}
Of course $\mathcal{A}_\tau$ is completely continuous, for all $\tau \in [0,1]$.
In particular, by $C^{1,\alpha}$ estimates in \cite{regularidade}, it follows that
$
\|\mathcal{A}_\tau\|_{E} \leq C \{ \,\|\mu_\tau\|_{L^{\infty}[0,1]} \, \|c\|_{L^p} \|u\|_\infty +\|f\|_{L^p}+1\, \} \leq  C_0 (1+\|u\|_{\infty}).
$
Now the conclusion is just a combination of topological arguments and Fredholm theory for the Laplacian operator, cf.\ Lemma 5.8, Proposition 5.9 and Theorem 2.4 in \cite{Arms2009}, over the space $E$.
\end{proof}

\section*{Acknowledgments}
Part of this work was done during the visit of the second author to the Pontifícia Universidade Católica do Rio de Janeiro. She would like to thank all the members of the Department of Mathematics for their warm hospitality.

\end{document}